\documentclass[11pt, a4paper,oneside]{amsart}

\usepackage[a4paper,inner=2.5cm,outer=2cm,top=3cm,bottom=3cm,pdftex]{geometry}
\usepackage[utf8]{inputenc}
\usepackage[T2A]{fontenc}
\usepackage{amssymb,amsmath}
\usepackage[american]{babel}
\usepackage{amsopn}
\usepackage{enumerate}
\usepackage{mathrsfs}
\usepackage{wasysym}
\usepackage{color}
\usepackage{fancyhdr}
\newtheorem{theorem}{Theorem}[section]
\newtheorem{lemma}[theorem]{Lemma}
\newtheorem{definition}[theorem]{Definition}

\newtheorem{proposition}[theorem]{Proposition}
\newtheorem{remark}[theorem]{Remark}

\newcommand\RR{{{\mathbb R}}}

\newcommand{\eps}{\varepsilon}

\newcommand{\reff}[1]{(\ref{#1})}

\newcommand{\p}{\partial}
\newcommand{\La}{\Delta}

\newcommand{\beq}{\begin{eqnarray*}}
\newcommand{\enq}{\end{eqnarray*}}
\newcommand{\beqq}{\begin{eqnarray}}
\newcommand{\enqq}{\end{eqnarray}}
\newcommand{\ben}{\begin{equation}\label}
\newcommand{\enn}{\end{equation}}
\newcommand{\bef}{\begin{proof}}
\newcommand{\enf}{\end{proof}}
\begin{document}

\title
{Cauchy problem for the incompressible
Navier-Stokes equation with an external force}

\author[ Di Wu]
{Di Wu}
\date{}
\address{\noindent \textsc{Di Wu, Institute de Math\'ematiques de Jussieu-Paris Rive Gouche UMR CNRS 7586, Univerist\'e Paris Diderot, 75205, Paris, France}}
\email{di.wu@imj-prg.fr}


\begin{abstract}
In this paper we focus on the Cauchy problem for the incompressible Navier-Stokes equation with a rough external force. If the given rough external force is small, we prove the local-in-time existence of this system for any initial data belonging to the critical Besov space $\dot{B}_{p,p}^{-1+\frac{3}{p}}$, where $3<p<\infty$. Moreover, We show the long-time behavior of the priori global solutins constructed by us. Also, we give three kinds of uniqueness results of the forced Navier-Stokes equations.
 
\end{abstract}

\maketitle

\section{Introduction}
We study the incompressible Navier-Stokes equations in $\RR^3$,
\beq
(NSf)\left\{
  \begin{array}{ll}
   \p_t u_f-\La u_f+u_f\cdot\nabla u_f=-\nabla p+f,\\
   \nabla\cdot u_f=0,\\
   u_f|_{t=0}=u_0.
  \end{array}
\right.
\enq 
Here $u_f$ is a three-component vector field $u_f=(u_{1,f},u_{2,f},u_{3,f})$ representing the velocity of the fluid, $p$ is a scalar denoting  the pressure, and both are unknown functions of the space variable $x\in\RR^3$ and of the time variable $t>0$. Finally $f=(f_1,f_2,f_3)$ denotes a given external force defined on $[0,T]\times\RR^3$ for some $T\in\RR_+\cup\{\infty\}$. We recall  the Navier-Stokes scaling : $\forall\lambda>0$, the vector field $u_f$ is a solution to $(NSf)$ with initial data $u_0$ if $u_{\lambda,f_{\lambda}}$ is a solution to $(NSf_{\lambda})$ with initial data $u_{0,\lambda}$, where
\beq
&u_{\lambda,f_{\lambda}}(t,x):=\lambda u_f(\lambda^2t,\lambda x),~f_{\lambda}(t,x):=\lambda^3 f(\lambda^2t,\lambda x),\\
& p_{\lambda}(t,x):=\lambda^2 p(\lambda^2 t,\lambda x)~~  \mathrm{and}~~u_{0,\lambda}:=\lambda u_0(\lambda x).
\enq 
Spaces which are invariant under the Navier-Stokes scaling are called critical spaces for the Navier-Stokes equation. Examples of critical spaces of initial data for the Navier-Stokes equation in 3D are:
\beq
 L^3(\RR^3)\hookrightarrow\dot{B}^{-1+\frac{3}{p}}_{p,q}(\RR^3)(p<\infty,q\leq\infty)\hookrightarrow\mathrm{BMO}^{-1}\hookrightarrow\dot{B}^{-1}_{\infty,\infty}
 \enq 
(See below for definitions).

To put our results in perspective, let us first recall related results concerning the Cauchy problem for the classical (the case $f\equiv0$) Navier-Stokes equation with possibly irregular initial data:
\beq
(NS)\left\{
  \begin{array}{ll}
   \p_t u-\La u+u\cdot\nabla u=-\nabla p,\\
   \nabla\cdot u=0,\\
   u|_{t=0}=u_0.
  \end{array}
\right.
\enq 
In the pioneering work \cite{L}, J. Leray introduced the concept of weak solutions to $(NS)$ and proved global existence for datum $u_0\in L^2$. However, their uniqueness has remained an open problem.
In 1984, T. Kato \cite{Kato} initiated the study of $(NS)$ with initial data belonging to the space $L^3(\RR^3)$ and obtained global existence in a subspace of $C([0,\infty),L^3(\RR^3))$ provided the norm $\|u_0\|_{L^3(\RR^3)}$ is small enough.
 The existence result for initial data small in the Besov space $\dot{B}^{-1+\frac{3}{p}}_{p,q}$ for $p,\in [1,\infty)$ and $q\in[1,\infty]$ can be found in \cite{CMP,gip}. The function spaces $L^3(\RR^3)$ and  $\dot{B}^{-1+\frac{3}{p}}_{p,q}$ for $(p,q)\in [1,\infty)^2$ both guarantee the existence of local-in-time solution for any initial data. In 2001, H. Koch and D. Tataru \cite{KT} showed that global well-posedness holds as well for small initial data in the space $\mathrm{BMO}^{-1}$. This space consists of vector fields whose components are derivatives of  $BMO$ functions.  On the other hand, it has been shown by J. Bourgain and N.  Pavlović \cite{BP} that the Cauchy problem with initial data in $\dot{B}^{-1}_{\infty,\infty}$ is ill-posed no matter how small the initial are. Also P. Germain showed the ill-posedness for initial data in $\dot{B}^{-1}_{\infty,q}$ for any $q>2$, see \cite{G}.  
 
However, the situation is more subtle when it comes to forced Navier-Stokes equations. In 1999, M. Cannone and F. Planchon \cite{CP} worked on constructing global mild solutions in $C([0,T),L^3(\RR^3))$ to the Cauchy problem for the Navier-Stokes equations with an external force. They showed the local-in-time wellposedness for any initial data $u_0\in L^3(\RR^3)$, if the external force $f$ can be written as $f=\nabla\cdot V$ and $\sup_{0<t<T}t^{1-\frac{3}{p}}\|V\|_{L^{\frac{p}{2}}}$ is small enough for some $3<p\leq 6$ and $T>0$. Also they showed there exists a unique global solution to $(NSf)$, provided $T=\infty$ and $u_0$ is small enough in $\dot{B}^{-1+\frac{3}{q}}_{q,\infty}$ with $3<q<\frac{3p}{6-p}$. Later in 2005, M. Cannone and G. Karch \cite{CK} proved that there exists a solution $u_f\in C_w(\RR_+,L^{3,\infty}(\RR^3))$ to $(NSf)$, if the initial data $u_0\in L^{3,\infty}$ is small enough and the external force $f$ satisfies that 
\beqq\label{fL3}
\sup_{t>0}\Big\|\int_0^t e^{(t-s)\La }\mathbb{P}fds\Big\|_{L^{3,\infty}}
\enqq 
is small enough depending on the norm of the bilinear operator $B$ (defined in \reff{Buv}) in $L^{\infty}(\RR_+,L^{3,\infty})$.

The basic approach to obtain the above results is, in principle, always the same. One first transforms the Navier-Stokes equations $(NSf)$ into an integral equation,
\beqq\label{mild-f}
u_f(t)=e^{t\La }u_0+\int_0^t e^{(t-s)\La }\mathbb{P}f(s)ds+B(u_f,u_f)(t)
\enqq
where
\beqq\label{Buv}
B(u,v):=-\frac{1}{2}\int_0^t e^{(t-s)\La }\mathbb{P}\nabla\cdot(u\otimes v+v\otimes u)ds,
\enqq  
$\mathbb{P}$ being the projection onto divergence free vector fields. It is customary to obtain the existence of a strong global $(T=\infty)$ or local $(T<\infty)$ solution $u_f\in X_T$ of \reff{mild-f}, with $X_T$ being an abstract critical Banach space, by means of the standard contraction lemma. For example, in \cite{CMP, CK} the terms $e^{t\La }u_0$ and $\int_0^t e^{(t-s)\La }\mathbb{P}f(s)ds$ are treated as the first point of the iteration and they require that $e^{t\La }u_0$, $\int_0^t e^{(t-s)\La }\mathbb{P}f(s)ds$  both belong to the corresponding Banach space $X_T$. That is why in \cite{CMP} $V$ needs to have a suitable decay in time and in \cite{CK} the smallness is measured in $L^{3,\infty}(\RR^3)$ and the initial data $u_0$ is restricted to $L^{3,\infty}$. The big difference between \cite{CP} and \cite{CK} is the following: in \cite{CP} the external force has good regularity and $e^{t\La }u_0$ belongs to Kato's space for initial data belonging to $\dot{B}^{s_p}_{p,\infty}$ for some $p>3$ (see Definition \ref{k}), which allows the fixed point lemma to work in Kato's space. Therefore the solutions in \cite{CP} belong to $C([0,T^*),L^3)$; however in \cite{CK}, the external force is rough, which limits the regularity of solution. Therefore in \cite{CK} the solutions to \cite{CK} only belong to $L^{\infty}_t(L^{3,\infty})$, even for small smooth initial data. That is the reason why these solutions  lack  uniqueness, unless the solution is small in $L^{\infty}_t(L^{3,\infty})$.

In this paper we consider $(NSf)$ with an external force given in \cite{CK}, however the class of initial data is different. More precisely, we consider the force $f$ satisfying : $f\in C(\RR_+,\mathcal{S}'(\RR^3))$ with for any $t>0$
\beq
\int_0^te^{(t-s)\La }\mathbb{P}fds\in L^{\infty}(\RR_+,L^{3,\infty}),
\enq 
which belongs to $C_w(\RR_+, L^{3,\infty}(\RR^3))$, see \cite{CK}. Under a smallness assumption on $f$ (controlled by a universal small positive constant depending on the singularity of initial data),
we first show 
 the local and global existence to $(NSf)$  for initial data $u_0$ belonging to $\dot{B}^{s_p}_{p,p}$. Moreover we obtain that the above solution belongs to $L^{\infty}_t(L^{3,\infty})$ when its initial data is in $L^{3,\infty}\cap\dot{B}^{s_p}_{p,p}$ for $p>3$.
 Then we show the long-time behavior and stability of the above priori global solutions with initial data in $L^{3,\infty}\cap\dot{B}^{s_p}_{p,p}$.
 We need to mention that the uniqueness of solutions in $L^{\infty}_t(L^{3,\infty})$, even for smooth initial data, is a still open problem. However we show that if the difference between the above solution and another solution to $(NSf)$ with the same initial data belongs to $C([0,T],L^{3,\infty})$ or has finite energy on some interval $[0,T]$, then they are equal on $[0,T]$.

The rest of the paper is organized as follows. In Section 2 we give some notations and the main results of this paper. Section 3 addresses the proof of the existence and uniqueness of solutions to $(NSf)$ with initial data $u_0$ belonging to $\dot{B}^{s_p}_{p,p}$. Section 4 is devoted to the long-time behavior and stability of a priori global solution to $(NSf)$ described in Section 2. The last section is devoted to  a regularity result via an iteration. In the appendix, we recall several known results and properties of solutions in Besov spaces. 
\section{Notations and Main Results}
Let us first recall the definition of Besov spaces, in dimension $d\geq1$.
\begin{definition}\label{besov}
	Let $\phi$ be a function in $\mathcal{S}(\RR^d)$ such that $\hat{\phi}=1$ for $|\xi|\leq 1$ and $\hat{\phi}=0$ for $|\xi|>2$, and define $\phi_j:=2^{dj}(2^{j}x)$. Then the frequency localization operators are defined by 
	\beq
	S_j:=\phi_j*\cdot,~~\La _j:=S_{j+1}-S_j.
	\enq 
	Let $f$ be in $\mathcal{S}'(\RR^d)$. We say $f$ belongs to $\dot{B}^{s}_{p,q}$ if 
	\begin{enumerate}
      \item the partial sum $\sum_{j=-m}^{m}\La_j f$ converges to $f$ as a tempered distribution if $s<\frac{d}{p}$ and after taking the quotient with polynomials if not, and
      \item 
      \beq
      \|f\|_{\dot{B}^{s}_{p,q}}:=\|2^{js}\|\La_j f\|_{L^p_x}\|_{\ell^q_j}<\infty.
      \enq 
    \end{enumerate}
\end{definition}

We refer to \cite{CL} for the introduction of the following type of space in the context of the Navier-Stokes equations.
\begin{definition}\label{time-besov}
	Let $u(\cdot,t)\in\dot{B}^{s}_{p,q}$ for a.e. $t\in (t_1,t_2)$ and let $\La_j$ be a frequency localization with respect to the $x$ variable (see Definition \ref{besov}). We shall say that $u$ belongs to $\mathcal{L}^\rho([t_1,t_2],\dot{B}^{s}_{p,q})$ if 
	\beq
	\|u\|_{\mathcal{L}^\rho([t_1,t_2],\dot{B}^{s}_{p,q})}:=\|2^{js}\|\La_j u\|_{L^{\rho}([t_1,t_2]L^p_x})\|_{\ell^q_j}<\infty.
	\enq 
\end{definition}
Note that for $1\leq\rho_1\leq q\leq\rho_2\leq\infty$, we have
\beq
L^{\rho_1}([t_1,t_2],\dot{B}^{s}_{p,q})\hookrightarrow\mathcal{L}^{\rho_1}([t_1,t_2],\dot{B}^{s}_{p,q})\hookrightarrow\mathcal{L}^{\rho_2}([t_1,t_2],\dot{B}^{s}_{p,q})\hookrightarrow L^{\rho_2}([t_1,t_2],\dot{B}^{s}_{p,q}).
\enq 
Let us introduce the following notation (also used in \cite{gkp}): we define $s_p:=-1+\frac{3}{p}$ and 
\beq
\mathbb{L}^{a:b}_{p,\bar{p}}(t_1,t_2):=\mathcal{L}^{a}([t_1,t_2];\dot{B}^{s_p+\frac{2}{a}}_{p,\bar{p}})\cap\mathcal{L}^{b}([t_1,t_2];\dot{B}^{s_p+\frac{2}{b}}_{p,\bar{p}}), \mathbb{L}^{a:b}_{p}(t_1,t_2):=\mathbb{L}^{a:b}_{p,\bar{p}}(t_1,t_2)
\enq 
\beq
\mathbb{L}^a_{p}(t_1,t_2):=\mathbb{L}^{a:a}_{p}(t_1,t_2),~\mathbb{L}^{a:b}_{p}(T):=\mathbb{L}^{a:b}_{p}(0,T)
\mathrm{and}~~\mathbb{L}^{a:b}_{p}[T<T^*]:=\cap_{T<T^*}\mathbb{L}^{a:b}_{p}(T).
\enq 
\begin{remark}\label{notation}
	We point out that according to our notations, $u\in\mathbb{L}^{a:b}_{p}[T<T^*]$ merely means that $u\in\mathbb{L}^{a:b}_{p}(T)$ for any $T<T^*$ and does not imply that $u\in\mathbb{L}^{a:b}_{p}(T^*)$ (the notation does not imply any uniform control as $T\nearrow T^*$).
\end{remark}
\begin{definition}\label{k}
   Let $p\geq 3$. Kato's space is defined as follow,
   \beq
   K_p:=\{u\in C(\RR_+,L^p(\RR^3)):\|u\|_{K_p}:=\sup_{t>0}t^{\frac{1}{2}-\frac{3}{2p}}\|u(t)\|_{L^p(\RR^3)}<\infty\}.
   \enq
	
\end{definition}

In this paper we are also interested in the weak-strong uniqueness of solutions to $(NSf)$. We introduce the following notations. We define that for any $T\in \RR_+\cup\{+\infty\}$
\beq
E(T)=L^{\infty}([0,T^*),L^2)\cap L^{2}([0,T^*),\dot{H}^1)
\enq  
and
\beq
E_{loc}(T)=L^{\infty}_{loc}([0,T),L^2)\cap L^{2}_{loc}([0,T),\dot{H}^1).
\enq

We also recall the definition of Lorentz spaces $L^{p,q}$ with $1<p<\infty$ and $1\leq  q\leq \infty$.
\begin{definition}
	Let $(X,\lambda)$ be a measure space. Let $f$ be a scalar-valued $\lambda$-measurable function and 
	\beq
	\lambda_f(s):=\lambda\{x: f(x)>s\}.
	\enq
	Then the re-arrangement function $f^*$ of $f$ is defined by:
	\beq
	f^*(t):=\inf\{s:\lambda_f(s)\leq t\}.
	\enq 
	And for any $1<p<\infty$, the Lorentz spaces $L^{p,q}$ is defined by:
	\beq
	L^{p,q}(\RR^d):=\{ f:\RR^d\to\mathbb{C},~~\|f\|_{L^{p,q}}<\infty\},
	\enq 
	where
	\beq
	\|f\|_{L^{p,q}}=\left\{
  \begin{array}{ll}
   \frac{q}{p}\Big[\int_0^{\infty}\big(t^{\frac{1}{p}}f^*(t)\big)dt\Big]^{\frac{1}{q}},~~q<\infty ,\\
   \sup_{t>0}\{t^{\frac{1}{p}}f^*(t)\},~~q=\infty.
  \end{array}
\right.
	\enq 
\end{definition}

We note that it is standard to use the above as a norm even if it does not satisfy the triangle inequality since one can find an equivalent norm that makes the space  into a Banach space.  In particular, $L^{p,\infty}$ agrees with the weak-$L^p$ (or Marcinkiewicz space):
\beq
L^{p^*}(\RR^d):=\{f:\RR^d\to\mathbb{C}:\|f\|_{L^{p^*}}<\infty\},
\enq 
which is equipped the following quasi-norm
\beq
\|f\|_{L^{p^*}}:=\sup_{t>0}t[\lambda_f(t)]^{\frac{1}{p}}.
\enq

To deal with external forces and for simplicity of notation we introduce the following space (introduced in \cite{CK}),
\beq 
\mathcal{Y}=\big\{f\in C(\RR_+,\mathcal{S}'(\RR^3)): \int_0^t e^{(t-s)\La }\mathbb{P}f(s)ds\in C_w(\RR_+,L^{3,\infty})\}
\enq 
equipped with the norm
\beq
\|f\|_{\mathcal{Y}}:=\sup_{t>0}\Big\|\int_0^t e^{(t-s)\La }\mathbb{P}f(s)ds\Big\|_{L^{3,\infty}}.
\enq

\begin{remark}
We mention that $\mathcal{Y}$ contains many rough external forces.

For example, 
	    \begin{enumerate}
	    	\item For every $g\in C_w(\RR_+,L^{\frac{3}{2},\infty})$, $f:=\nabla g\in \mathcal{Y}$ and $\|f\|_{Y}\lesssim \|g\|_{L^{\infty}(\RR_+,L^{\frac{3}{2}\infty})}$ (see Lemma 3.2 in \cite{CK}).
	    	\item Every time-independent $f$ satisfying $\La ^{-1}f\in L^{3,\infty}$ belongs to $\mathcal{Y}$ (see Theorem 4.3 in \cite{bbis}).
	    	\item By Lemma 3.4 in \cite{CK}, $\mathcal{Y}$ contains some really rough external force: $f:=(c_1\delta_0,c_2\delta_0,c_3\delta_0)$, where $\delta_0$ stands for the dirac mass.
	    \end{enumerate}
	    \end{remark}
According to Theorem 2.1 in \cite{CK}, there exists a constant $\eps_0>0$ such that if $f\in \mathcal{Y}$ and  $u_0\in  L^{3,\infty}$ satisfy that  $\|u_0\|_{L^{3,\infty}}+\|f\|_{\mathcal{Y}}<\eps_0$, there exists a unique solution to $(NSf)$ with initial data $u_0$ and external force $f$, denoted by $NSf(u_0)$, which belongs to $C_w(\RR_+,L^{3,\infty})$ such that 
\beq
\|NSf(u_0)\|_{L^{\infty}(\RR_+,L^{3,\infty})}\leq 2(\|u_0\|_{L^{3,\infty}}+\|f\|_{\mathcal{Y}}). 
\enq 
In particular, we have $NSf(0)\in C_w(\RR_+,L^{3,\infty})$ satisfying 
\beq
\|NSf(0)\|_{L^{\infty}(\RR_+,L^{3,\infty})}\leq 2\|f\|_{\mathcal{Y}}<2\eps_0. 
\enq 
From now on, we denote  $U_f:=NSf(0)$.

Now let us state our main results. We first state a local in time existence result for $(NSf)$ for initial data belonging to $\dot{B}^{s_p}_{p,p}$ for any $p>3$ under a smallness  assumption on $f$ depending on $p$ (it is no loss of generality to set $p,p$ rather than $p,q$, which deduces some technical difficulties). Moreover we obtain a local in time existence result for $(NSf)$ in $L^{\infty}_t(L^{3,\infty})$ for initial data belonging to $\dot{B}^{s_p}_{p,p}\cap L^{3,\infty}$.

\begin{theorem}[Existence]\label{NSexistence}
     Let $p>3$.
     There exists a  small universal constant $c(p)>0$ with the following properties:
     
     Suppose that $f\in \mathcal{Y}$ is a given external force satisfying that $\|f\|_{\mathcal{Y}}<c(p)$. Then 
     \begin{enumerate}
     \item 
     
     for any initial data $u_0\in \dot{B}^{s_p}_{p,p}$, a unique  maximal time $T^*(u_0,f)>0$ and a unique solution $u_f\in\mathbb{L}^{r_0:\infty}_p[T<T^*]+ C_w([0,T^*),L^{3,\infty})$ to $(NSf)$ with initial data $u_0$ exist such that 
    \beq
    u_f-U_f\in \mathbb{L}^{r_0:\infty}_p[T<T^*].
    \enq
    with $r_0=\frac{2p}{p-1}$.
    And if $T^*<\infty$, 
    \beq
    \limsup_{T\to T^*}\|u-U_f\|_{\mathbb{L}^{r_0:\infty}_{p}(T)}=\infty.
    \enq          
    Moreover there exists a small constant $\eta>0$ depending on $f$ and $p$ such that 
	\beq
	&&\|u_0\|_{\dot{B}^{s_p}_{p,p}}<\eta\Rightarrow T^*=\infty~~\mathrm{and}~~\|u_f-U_f\|_{\mathbb{L}^{r_0:\infty}_p(\infty)}\leq C(f)\|u_0\|_{\dot{B}^{s_p}_{p,p}}.
	\enq	
	\item  if $u_0\in\dot{B}^{s_p}_{p,p}\cap L^{3,\infty}$, the above solution $u_f$ to $(NSf)$ with initial data $u_0$ belongs to \\
	$C_w([0,T^*),L^{3,\infty})$.
	\item if $u_0\in \dot{B}^{s_p}_{p,p}\cap L^{2}(\RR^3)$, the above solution $u_f$ to $(NSf)$ with initial data $u_0$ satisfies that
         \beq
         u_f-U_f\in E_{loc}(T^*).
         \enq 
	\end{enumerate}
	\end{theorem}
	
Our method is to transform  $(NSf)$ into  the perturbation equation,
\beq
(PNS_{U_f})\left\{
  \begin{array}{ll}
  \p_t v-\La v+v\cdot\nabla v+U_f\cdot\nabla v+v\cdot\nabla U_f=-\nabla \pi,\\
  \nabla\cdot v=0,\\
  v|_{t=0}=v_0:=u_0,
  \end{array}
\right.
\enq 
The corresponding integral form of $(PNS_{U_f})$ is
\beqq\label{vBLU}
v=e^{t\La }v_0+B(v,v)+2B(U_f,v),
\enqq 
where $B$ is defined as \reff{Buv}. The reason why we focus on $(PNS_{U_f})$ is that \reff{vBLU} allows  us to use the classical contraction lemma in the Besov space $\mathcal{L}^r([0,T], \dot{B}^{-1+\frac{3}{p}+\frac{2}{r}}_{p,p})$ with any $p>3$ and some $r>2$.

Also in order to control the energy estimate, we adopt the argument about the trilinear form $\int_0^T\int_{\RR^3}(a\cdot\nabla b)\cdot c(t)dxdt$ in \cite{GP}.

From Theorem \ref{NSexistence}, for any $u_0\in \dot{B}^{s_p}_{p,p}\cap L^{3,\infty}$, there exists a solution $u_f\in C_w([0,T^*),L^{3,\infty})$. Actually $C_w([0,T^*),L^{3,\infty})$ is the highest regularity of solutions to $(NSf)$, as the singularity of $f$ limits it. Therefore the uniqueness of solutions to $(NSf)$ in $C_w([0,T^*),L^{3,\infty})$ is crucial. 	
We  point out that we cannot prove that the above solution is unique in $L^{\infty}_t(L^{3,\infty})$ without the smallness assumption on the solution. Actually even if for $(NS)$ the uniqueness in $L^{\infty}_t(L^{3,\infty})$ is still open (the uniqueness just holds when solution is small in $L^{\infty}_t(L^{3,\infty})$). However, we obtain that the above solution is unique in the following sense: 

\begin{theorem}[Uniqueness]\label{uniquness}
    Let $p>3$.
	There exists a universal small constant $0<c_1(p)\leq c(p)$ with the following properties:
	
	Suppose that $f\in \mathcal{Y}$ is a given external force satisfying that $\|f\|_{\mathcal{Y}}<c_1$ and $u_f\in C_w([0,T^*),L^{3,\infty})$ is a solution to $(NSf)$ constructed in Theorem \ref{NSexistence} with initial data $u_0\in L^{3,\infty}\cap \dot{B}^{s_p}_{p,p}$. Then $u_f$ is unique in the following sense: Assume that $\bar{u}_f\in C_w([0,T],L^{3,\infty})$ for some $T<T^*$ is another solution to $(NSf)$ with  same initial data $u_0$.
	\begin{itemize}
		\item  If 
	          $ 
	        u_f-\bar{u}_f\in \mathbb{L}^{r:\infty}_p(T)
	          +\{U(t,x)\in C_{w}(\RR_+,L^{3,\infty}):\|U\|_{L^{\infty}(\RR_+),L^{3,\infty}}<2c_1\}
	          $ for some $2<r<\frac{2p}{p-3}$,
	          then $u_f\equiv \bar{u}_f$ on $[0,T]$.
	    \item if $u_f-\bar{u}_f\in C([0,T],L^{3,\infty})$, then $u_f\equiv \bar{u}_f$ on $[0,T]$
	    \item  if $3<p<5$ and 
	$
	u_f-\bar{u}_f\in L^{\infty}([0,T],L^2)\cap L^{2}([0,T],\dot{H}^1),
	$
	then $u_f\equiv \bar{u}_f$ on $[0,T]$.
	\end{itemize}
\end{theorem}
 We prove Theorem \ref{NSexistence} and \ref{uniquness} in Section 3. Our method depends on an iteration regularity result developed in Section 5.

The global existence of the solutions described in Theorem \ref{NSexistence} for large initial data $u_0\in \dot{B}^{s_p}_{p,p}$ is still open, even for $f=0$. 

We mention that even if a solution $u_f\in C_w([0,T^*),L^{3,\infty})$ to $(NSf)$  is global, which just means its corresponding life span $T^*=\infty$, one cannot obtain that $u_f(t)$ has a uniform bound in $L^{3,\infty}$ as $t$ goes to infinity in general. However, if $u_f$ is a global solution to $(NSf)$ with initial data $u_0\in \dot{B}^{s_p}_{p,p}\cap L^{3,\infty}$ described in Theorem \ref{NSexistence}, the next theorem shows the solution belongs to $L^{\infty}(\RR_+,L^{3,\infty})$.

Comparing with previous results of long-time behavior, our assumptions on $u_f$ and $f$ are all in critical spaces, but the class of initial data is smaller. For example, in \cite{bbis} C. Bjorland, L. Brandolese, D. Iftimie \& M. E. Schonbek proved that if the external force $f$ is time-independent satisfying that $\La^{-1}f\in L^{3,\infty}\cap L^4$ and $\|\La^{-1}f\|_{L^{3,\infty}}$ is small, then for any priori global solution $u_f\in C_w(\RR_+,L^{3,\infty})\cap L^4_{loc}(\RR_+,L^4)$ with initial data $u_0=v_0+w_0$ satisfying that $v_0\in L^2$ and $\|w_0\|_{L^{3,\infty}}$ is smaller than a fixed small constant $\epsilon$, then $u_f\in L^{\infty}(\RR_+, L^{3,\infty})$. It clear that the space of initial data they are working on is larger than $L^{3,\infty}\cap\dot{B}^{s_p}_{p,p}$ and $\La^{-1}f\in L^{3,\infty}\cap L^4$ implies that $f\in\mathcal{Y}$. However the condition $\La^{-1}f\in L^{3,\infty}\cap L^4$ excludes some important singular force: $\La^{-1}f\sim \frac{1}{|x|}$, which belongs to $\mathcal{Y}$.

\begin{theorem}[Long-time behavior of global solutions]\label{long-time-u}
Let $p>3$.
Suppose that  $f\in\mathcal{Y}$ is a given external force such that $\|f\|_{\mathcal{Y}}<c_1(p)$, where $c_1(p)$ is the small constant in Theorem \ref{uniquness}.

Suppose that $u_f\in C_w([0,\infty),L^{3,\infty})$ is an priori global solution to $(NSf)$ described in Theorem \ref{NSexistence}, whose initial data $u_0\in L^{3,\infty}\cap\dot{B}^{s_p}_{p,p}$. Then 
 there exists a constant $M$ independent of $u_f$ such that 
	    \beq 
       \limsup_{t\to\infty}\|u_f(t)\|_{L^{3,\infty}}\leq M.
       \enq 
	
\end{theorem}

The idea of the proof of  long-time behavior, as in \cite{gip,bbis}, consists in decomposing the initial velocity in a small part plus a square integrable part. The small part remains small by the small data theory and the square-integrable part will become small at some point by using some energy estimates. More precisely, we split the initial data $u_0=\bar{u}_0+v_0$, where $\bar{u}_0$ is small enough in $L^{3,\infty}$ and $v_0\in L^2(\RR^3)\cap L^{3,\infty}$. By the global existence of $(NSf)$ for small initial data (see \cite{CK}) we have $NSf(\bar{u}_0)\in L^{\infty}(\RR_+,L^{3,\infty})$ and $v:=u_f-NSf(\bar{u}_0)$ satisfies the perturbation equation $PNS_{NSf(\bar{u}_0)}$. Compared to the unforced case, it is hard to obtain that $v$ has finite energy on $[0,T]$ for any $0<T<\infty$ in general, which is the reason why the restriction on external force: $\La^{-1}f\in L^{3,\infty}\cap L^4$ is crucial in Theorem 4.7 in \cite{bbis}. In our case, we have obtained that $v$ has finite energy on $[0,T]$ for any $0<T<\infty$ by Theorem \ref{NSexistence}.

We show the stability of priori global solutions constructed in Theorem \ref{NSexistence} in the following theorem.

\begin{theorem}[Stability of global solutions]\label{stability-u}
Let $p>3$.
	Suppose that $f\in C(\RR_+,\mathcal{S}'(\RR^3))$ satisfies the same conditions as Theorem \ref{long-time-u} and that $u_f$ is an priori global solution to $(NSf)$ described in Theorem \ref{NSexistence} with initial data $u_0\in\dot{B}^{s_p}_{p,p}$.
	
  Then there exists an $\delta$ (depending on $u_f$) with the following property.\\
  For any initial data $\bar{u}_0\in \dot{B}^{s_p}_{p,p}$ satisfying $\|u_0-\bar{u}_0\|_{\dot{B}^{s_p}_{p,p}}<\delta$, there exist a global solution $\bar{u}_f$ to $(NSf)$ with initial data $\bar{u}_0$, and
  \beq
  \|u_f(t)-\bar{u}_f(t)\|_{\mathbb{L}^{r_0:\infty}_{p}(\infty)}\lesssim\|u_0-\bar{u}\|_{\dot{B}^{s_p}_{p,p}}.
  \enq 
\end{theorem}

The stability result for the solution introduced as above is an extension of Theorem 3.1 in \cite{gip}. We prove it with a similar proof to Theorem 3.1 in \cite{gip},  the difference between these two cases is that there is a small bounded in time and no-decay in time  drift part in our case. The proofs of Theorem \ref{long-time-u} and \ref{stability-u} are presented in Section 4.

\section{Existence and uniqueness of $(NSf)$}
The aim of this section is to prove Theorem \ref{NSexistence} and Theorem \ref{uniquness}. Let us recall the situation: Let $p>3$ be fixed and
 the external force $f\in\mathcal{Y}$ and $\|f\|_{\mathcal{Y}}<c(p)$, where $c(p)$ is a  small universal constant smaller than the constant $\eps$ in Theorem 2.1 of \cite{bbis}. The class of initial data is $\dot{B}^{s_p}_{p,p}$. 
\subsection{Existence of $(NSf)$}
 By Theorem 2.1 in \cite{bbis}, there exists a unique solution $U_f:=NSf(0)\in L^{\infty}(\RR_+,L^{3,\infty})$ such that 
 \beqq\label{small-nsf(0)}
 \|U_f\|_{L^{\infty}(\RR_+,L^{3,\infty})}\leq2\|f\|_{\mathcal{Y}}<2c(p).
 \enqq 

Then we can transform the Cauchy problem of $(NSf)$ into the Cauchy problem of $(PNS_{U_f})$:
\beq
(PNS_{U_{f}})\left\{
  \begin{array}{ll}
  \p_t v-\La v+v\cdot\nabla v+U_{f}\cdot\nabla v+v\cdot\nabla U_{f}=-\nabla \pi,\\
  \nabla\cdot v=0,\\
  v|_{t=0}=u_0,
  \end{array}
\right.
\enq 
whose integral form is 
\beq
v(t,x)=e^{t\La }u_0+B(v,v)+B(2U_f,v),
\enq
where $B$ is defined in \reff{Buv}. We use a standard fixed point lemma to solve the above system: We first recall without proofs the following lemma.
\begin{lemma}\label{fixed point}
	Let $X$ be a Banach space, $L$ a linear operator from $X\to X$ such that a constant $\lambda<1$ exists such that 
	\beq
	\forall x\in X,~~\|L(x)\|_X\leq\lambda\|x\|_X,
	\enq 
	and $B$ a bilinear operator such that for some $\gamma$, 
	\beq
	\forall(x,y)\in X^2,~~\|B(x,y)\|_X\leq\gamma\|x\|_X\|y\|_X.
	\enq
	Then for all $x_1\in X$ such that 
	\beq
	\|x_1\|_X<\frac{(1-\lambda)^2}{4\gamma},
	\enq
	the sequence defined by 
	\beq
	x^{(n+1)}=x_1+L(x^{(n)})+B(x^{(n)},x^{(n)})
	\enq 
	with $x^{(0)}=0$ converges in $X$ and towards the unique solution of 
	\beq
	x=x_1+L(x)+B(x,x)
	\enq 
	such that 
	\beq
	2\gamma\|x\|_X\leq(1-\lambda).
	\enq 
\end{lemma}
In the proof of Theorem \ref{NSexistence}, we first show the local in time existence of $(NSf)$ with initial data in $\dot{B}^{s_p}_{p,p}$. Next, we show the propagation of the regularity of the solution constructed above with initial data, in addition, belonging to $L^{3,\infty}$ or $L^2$.  
\begin{proof}[Proof of Theorem \ref{NSexistence}]
     Let $u_0\in\dot{B}^{s_p}_{p,p}$ be a divergence-free vector field. 
     We note that $v$ is the solution satisfying system $(PNS_{U_f})$ with initial data $u_0$. 
	
	\textbf{Existence:}
	It is clear that if there exists a solution $v$ to $(PNS_{U_f})$ with initial data $u_0$ on $[0,T]$, then $v+U_f$ is a solution to $(NSf)$ with initial data $u_0$. Hence to prove the first statement in Theorem \ref{NSexistence}, it is enough to prove that for any initial data $u_0\in \dot{B}^{s_p}_{p,p}$, there exists a unique $T^*>0$ and a unique solution $v\in\mathbb{L}^{r_0:\infty}_p$ to $(PNS_{U_f})$ with initial data $u_0$. 
	
	Now we start to prove the above statement by applying Lemma \ref{fixed point}.
	
	We choose $\mathcal{L}^{r_0}([0,T],\dot{B}_{p,p}^{s_+\frac{2}{r_0}})$ as the Banach space in Lemma \ref{fixed point}, where 
	$
	r_0=\frac{2p}{p-1}.
	$ 
	It is easy to check that $s_p+\frac{2}{r_0}>0$. To apply Lemma \ref{fixed point}, we need to obtain that $B(u,v)$ defined  in \reff{Buv} is a continuous bilinear operator from $\mathcal{L}^{r_0}([0,T],\dot{B}_{p,p}^{s_+\frac{2}{r_0}})\times\mathcal{L}^{r_0}([0,T],\dot{B}_{p,p}^{s_+\frac{2}{r_0}})$ to $\mathcal{L}^{r_0}([0,T],\dot{B}_{p,p}^{s_+\frac{2}{r_0}})$ and the linear operator $L(v):=B(2U_f,v)$ is continuous on $\mathcal{L}^{r_0}([0,T],\dot{B}_{p,p}^{s_+\frac{2}{r_0}})$ with its norm strictly smaller than $1$.
	
	In fact,  according to the first statement in Lemma \ref{Heat} and the first statement of  Proposition \ref{bilinear-besov},
    we have that $B$ is a continuous operator from $\mathcal{L}^{r_0}([0,T];\dot{B}^{s_p+\frac{2}{r_0}}_{p,p})\times \mathcal{L}^{r_0}([0,T];\dot{B}^{s_p+\frac{2}{r_0}}_{p,p})$ to $\mathcal{L}^{r_0}([0,T],\dot{B}_{p,p}^{s_+\frac{2}{r_0}})$ and hence, for some $\gamma>0$
    \beq
    \|B(v_1,v_2)\|_{\mathcal{L}^{r_0}([0,T];\dot{B}^{s_p+\frac{2}{r_0}}_{p,p})}\leq \gamma\|v_1\|_{\mathcal{L}^{r_0}([0,T];\dot{B}^{s_p+\frac{2}{r_0}}_{p,p})}\|v_2\|_{\mathcal{L}^{r_0}([0,T];\dot{B}^{s_p+\frac{2}{r_0}}_{p,p})}.
    \enq

      	According to the third statement in Proposition \ref{bilinear-besov}, replacing $w$ by $U_f$, we have
	 for any $v\in\mathcal{L}^{r_0}([0,T];\dot{B}^{s_p+\frac{2}{r_0}}_{p,p })$,
	\begin{equation*}
     \begin{aligned}
	\|B(2U_f,v)&\|_{\mathcal{L}^{r_0}([0,T];\dot{B}^{s_p+\frac{2}{r_0}}_{p,p})}\leq \|B(2U_f,v)\|_{\mathcal{L}^{r_0}([0,T];\dot{B}^{s_{\bar{p}}+\frac{2}{r_0}}_{\bar{p},p})} \\
	&\leq2C(p)\|U_f\|_{L^{\infty}(\RR_+,L^{3,\infty})}\|v\|_{\mathcal{L}^{r_0}([0,T];\dot{B}^{s_p+\frac{2}{r_0}}_{p,p})},
	\end{aligned}
	\end{equation*}
	where $C(p)\to\infty$ as $p\to\infty$ and $\frac{1}{\bar{p}}=\frac{1}{3}+\frac{1}{6p}$.
	By taking $c(p)\leq(4C(p))^{-1}$, then by the above estimate and \reff{small-nsf(0)} , we have
	\beq
	\lambda:=2\gamma_1\bar{C}(p)\|U_f\|_{L^{\infty}(\RR_+,L^{3,\infty})}<1,
	\enq 
	and
	\beq
	\|B(2U_f,v)\|_{\mathcal{L}^{r_0}([0,T];\dot{B}^{s_p+\frac{2}{r_0}}_{p,p})}\leq\lambda \|v\|_{\mathcal{L}^{r_0}([0,T];\dot{B}^{s_p+\frac{2}{r_0}}_{p,p})}.
	\enq 
	Therefore according to Lemma \ref{fixed point} and the fact that 
	\beq
	\|e^{t\La }u_0\|_{\mathcal{L}^{r_0}(\RR_+,\dot{B}^{s_p+\frac{2}{r_0}}_{p,p})}\lesssim\|u_0\|_{\dot{B}^{s_p}_{p,p}},
	\enq 
	one can find a small enough number $\eta(p,f)$ such that, for any $u_0\in\dot{B}^{s_p}_{p,p}$ with $\|u_0\|_{\dot{B}^{s_p}_{p,p}}<\eta$, there exists a unique global solution $v\in\mathcal{L}^{r_0}(\RR_+,\dot{B}^{s_p+\frac{2}{r_0}}_{p,p})$ with initial data $u_0$ satisfying that 
	\beq
	\|v\|_{\mathcal{L}^{r_0}(\RR_+,\dot{B}^{s_p+\frac{2}{r_0}}_{p,p})}\leq \frac{1-\lambda}{2\gamma}.
	\enq 
	Moreover we notice that for any given $u_0\in \dot{B}^{s_p}_{p,p}$ and any $T>0$,
	\beq
	&&\|e^{t\La}u_0\|_{\mathcal{L}^{r_0}([0,T];\dot{B}^{s_p+\frac{2}{r_0}}_{p,p})}=
	\Big(\sum_{j\in\mathbb{Z}}\big(2^{j(s_p+\frac{2}{r_0})}\|\La_j e^{t\La }u_0\|_{L^{r_0}([0,T];L^p_x)}\big)^p\Big)^{\frac{1}{p}}\\
	&=&\big\|\big(1-e^{-r_0Tc_p2^{2j}}\big)^{\frac{1}{r_0}}2^{js_p}\|\La_j u_0\|_{L^p}\big\|_{\ell^p}.
	\enq 
	Next, an application of Lebesgue's dominated convergence theorem shows that 
	\beq
	\lim_{t\to0}\big\|\big(1-e^{-r_0Tc_p2^{2j}}\big)^{\frac{1}{r_0}}2^{js_p}\|\La_j u_0\|_{L^p}\big\|_{\ell^p}=0.
	\enq 
	It follows that for any given $u_0\in \dot{B}^{s_p}_{p,p}$, there exists $T_0$ such that 
	\beq
	\|e^{t\La}u_0\|_{\mathcal{L}^{r_0}([0,T_0];\dot{B}^{s_p+\frac{2}{r_0}}_{p,p})}<\frac{(1-\lambda)^2}{4\gamma}.
	\enq 
	Therefore we have $v\in\mathcal{L}^{r_0}([0,T_0];\dot{B}^{s_p+\frac{2}{r_0}}_{p,p})$.

   Hence for any $u_0\in \dot{B}^{s_p}_{p,p}$, there exists a $T^*(u_0,f)>0$ such that $v\in\mathcal{L}^{r_0}([0,T^*),\dot{B}^{s_p}_{p,p})$. And according to Lemma \ref{Heat}, we obtain that $v\in\mathcal{L}^{r}([0,T^*);\dot{B}^{s_p+\frac{2}{r}}_{p,p})$ for any $r\in[r_0,\infty]$, which implies that $v\in\mathbb{L}^{r_0:\infty}_p[T<T^*]$.

   When $T^*<\infty$, we claim that 
    \beq
    \lim_{T\to T^*}\|v\|_{\mathbb{L}^{r_0:\infty}_{p}(T)}=\infty,
    \enq  
    by a similar argument in \cite{C1}. In fact, if 
    \beq
    \lim_{T\to T^*}\|v\|_{\mathbb{L}^{r_0:\infty}_{p}(T)}<\infty,
    \enq 
    in particular,
    \beq
    v\in \mathcal{L}^{\infty}([0,T^*),\dot{B}^{s_p}_{p,p})
    \enq 
    which implies that for any $\eps>0 $, there exists $N(\eps)$ such that for any $t'\in[0,T^*)$
    \beq
    \Big(\sum_{|j|>N(\eps)}2^{pjs_p}\|\La_j v(t')\|^{p}_{L^{p}}\Big)^{\frac{1}{p}}<\eps.
    \enq 
    Therefore for any fixed $t'\in [0,T^*)$,
    \beq
    &&\|e^{t\La }v(t')\|_{\mathcal{L}^{r_0}([0,T];\dot{B}^{s_p+\frac{2}{r_0}}_{p,p})}=\big\|\big(1-e^{-r_0Tc_p2^{2j}}\big)^{\frac{1}{r_0}}2^{js_p}\|\La_j v(t')\|_{L^p}\big\|_{\ell^p}\\
    &\lesssim&\Big(\sum_{|j|>N(\eps)}2^{pjs_p}\|\La_j v(t')\|^p_{L^p}\Big)^{\frac{1}{p}}+2^{N(\eps)s_p}(1-e^{r_0T})\|v\|_{\mathcal{L}^{\infty}([0,T^*),\dot{B}^{s_p}_{p,p})}\\
    &\lesssim&\eps+2^{N(\eps)s_p}(1-e^{r_0T})\|v\|_{\mathcal{L}^{\infty}([0,T^*),\dot{B}^{s_p}_{p,p})},
    \enq 
    which implies for any $t'\in [0,T^*)$, there exists a $\tau$ independent of $t'\in[0,T^*)$ such that 
    \beq
    \|e^{t\La}v(t')\|_{\mathcal{L}^{r_0}([0,\tau];\dot{B}^{s_p+\frac{2}{r_0}}_{p,p})}<\frac{(1-\lambda)^2}{4\gamma}.
    \enq 
    Hence we obtain that $v\in \mathcal{L}^{r_0}([0,T^*+\tau/2],\dot{B}_{p,p}^{s_p+\frac{2}{r_0}})$, which contradicts the maximality of $T^*$.
    
    To finish the proof of the first statement in Theorem \ref{NSexistence}, we need to prove $v$ is the unique solution to $(PNS_{U_f})$ with initial data $u_0\in\dot{B}^{s_p}_{p,p}$ in $\mathbb{L}^{r_0:\infty}_p[T<T^*]$. We suppose that $\bar{v}\in \mathbb{L}^{r_0:\infty}_p(T)$ for some $T<T^*$ is another solution to $(PNS_{U_f})$ with the same initial data $u_0$ and set $w:=\bar{v}-v$. It is easy to check that $w$ satisfies that 
    \beq
    w=B(w,w)+B(2(U_f+v),w).
    \enq 
    A similar argument as above implies that 
    \beq
    \begin{aligned}
    \|w\|_{\mathcal{L}^{r_0}_t(\dot{B}_{p,p}^{s_p+\frac{2}{r_0}})}\leq& K_0 \|w\|^2_{\mathcal{L}^{r_0}_t(\dot{B}_{p,p}^{s_p+\frac{2}{r_0}})}+K_0\|v\|_{\mathcal{L}^{r_0}_t(\dot{B}_{p,p}^{s_p+\frac{2}{r_0}})}\|w\|_{\mathcal{L}^{r_0}_t(\dot{B}_{p,p}^{s_p+\frac{2}{r_0}})}\\
    &+\lambda\|w\|_{\mathcal{L}^{r_0}_t(\dot{B}_{p,p}^{s_p+\frac{2}{r_0}})},
    \end{aligned}
    \enq 
    for some $K_0>0$. This fact implies that one can find a $K_1>K_0>0$ such that
    \beq
    \|w\|_{\mathcal{L}^{r_0}_t(\dot{B}_{p,p}^{s_p+\frac{2}{r_0}})}\leq K_1 \|w\|^2_{\mathcal{L}^{r_0}_t(\dot{B}_{p,p}^{s_p+\frac{2}{r_0}})}+K_1\|v\|_{\mathcal{L}^{r_0}_t(\dot{B}_{p,p}^{s_p+\frac{2}{r_0}})}\|w\|_{\mathcal{L}^{r_0}_t(\dot{B}_{p,p}^{s_p+\frac{2}{r_0}})}.
    \enq 
    We infer that 
    \beqq\label{wK}
    \|w\|_{\mathcal{L}^{r_0}_t(\dot{B}_{p,p}^{s_p+\frac{2}{r_0}})}(K_1 \|w\|_{\mathcal{L}^{r_0}_t(\dot{B}_{p,p}^{s_p+\frac{2}{r_0}})}+K_1\|v\|_{\mathcal{L}^{r_0}_t(\dot{B}_{p,p}^{s_p+\frac{2}{r_0}})}-1)\geq 0.
    \enqq 
    By continuity of the norm of $\mathcal{L}^{r_0}_t(\dot{B}_{p,p}^{s_p+\frac{2}{r_0}})$ with respect to the time, there exists $\tilde{T}$ such that for all $t\in[0,\tilde{T}]$
    \beq
    K_1 \|w\|_{\mathcal{L}^{r_0}_t(\dot{B}_{p,p}^{s_p+\frac{2}{r_0}})}+K_1\|v\|_{\mathcal{L}^{r_0}_t(\dot{B}_{p,p}^{s_p+\frac{2}{r_0}})}-1<0.
    \enq 
    Therefore, for $t\in[0,\tilde{T}]$ relation \reff{wK} can hold only if $\|w\|_{\mathcal{L}^{r_0}_{\tilde{T}}(\dot{B}_{p,p}^{s_p+\frac{2}{r_0}})}=0$, that is $w\equiv0$ on $[0,\tilde{T}]$, by continuity again, $w\equiv0$ on $[0,T]$ for any $T<T^*$.
    
    The first statement of the theorem is proved.
   
    \textbf{Propagation of perturbations:}
    
    Next we turn to show the propagation of $v$. According to Theorem \ref{Regu} by choosing $w=U_f$ and $\bar{w}=0$, we have that $v$ can be written as, for any $T\in[0,T^*)$
    \beq
    v=v^H+v^S,
    \enq 
    where $v^H=H_{N_0}\in \mathbb{L}^{1:\infty}_p(\infty)$  and $v^S=W_{N_0}+Z_{N_0}\in L^{\infty}([0,T],L^{3,\infty})$ with $N_0$ being the largest integer such that $3(N_0-1)<p$.   
    We first notice that in the case when $\bar{w}=0$, $H_{N}$ is a sum of a finite number of multilinear operators of order at most $N-1$, acting on $e^{t\La }u_0$ only. Hence according to Lemma \ref{weak-L^3-continuous} and an inductive argument, we obtain for any $N\geq 2$,
    \beq
    H_N\in L^{\infty}(\RR_+,L^{3,\infty}),
    \enq 
    which implies that $v^H\in L^{\infty}(\RR_+,L^{3,\infty})$.
    
    To prove the second statement of the theorem, we are left with the proof of $v\in C_w([0,T^*),L^{3,\infty})$. We notice that by Lemma 2 \& 3 in \cite{B1}, $e^{t\La }u_0\in C_w([0,\infty),L^{3,\infty})$. This fact combined with Lemma \ref{weak-L^3-continuous} implies that for any $T\in[0,T^*)$
    \beq
    v=v^H+v^S=H_{N_0}+ W_{N_0}+Z_{N_0}\in C_w([0,T^*),L^{3,\infty}).
    \enq 
    The second statement of Theorem \ref{NSexistence} is proved.

    \textbf{Finite energy of perturbations:}
     
    In the last part of the proof, we show that $v$ has finite energy on $[0,T]$ for any $T<T^*$, if $u_0\in \dot{B}^{s_p}_{p,p}\cap L^2$. 
    
    Now we suppose that $u_0\in \dot{B}^{s_p}_{p,p}\cap L^2$ and $T\in [0,T^*)$ is fixed. We recall that 
    \beq
    v=e^{t\La }u_0+B(v,v)+B(2U_f,v).
    \enq 
    It is clear that $e^{t\La }u_0\in E(\infty)$. Hence we only need to prove $B(v,v)+B(2U_f,v)\in E(T)$.
    
    By replacing $v$ of $B(v,v)+B(2U_f,v)$ by $v^H+v^S$, we have
    \beq
    \begin{aligned}
    	B(v,v)+B(2U_f,v)=&B(v^H,v^H+ 2v^S+2U_f)\\
    	&+B(v^S,v^S+2U_f).
    \end{aligned}
    \enq 
    By applying Lemma \ref{heat-energy} and the fact that $e^{t\La }u_0\in E(\infty)$, we first obtain $v^H=H_{N_0}\in E(\infty)$. Again by Lemma \ref{heat-energy}, we obtain that 
    \beq
    B(v^H,v^H+ 2v^S+2U_f)\in E(T),
    \enq 
    provided that $v^H\in\mathbb{L}^{1:\infty}_{p}(\infty)$ and $v^S+U_f\in \mathcal{L}^{\infty}([0,T],\dot{B}^{s_q}_{q,\infty})$ where $q=\frac{3p}{p-2}$.
   
   Now we turn to the proof of $B(v^S,v^S+2U_f)\in E(T)$. We recall that 
   \beqq\label{wzu}
   v^S+2U_f\in L^{\infty}([0,T],L^{3,\infty}).
   \enqq 
   On the other hand, by  $v^S\in \mathbb{L}^{r_0:\infty}_{\tilde{p},p}(T)$ with some $r_0=\frac{2p}{p-1}$ and some $\bar{p}<3$, we have
   \beq
   v^S\in \mathbb{L}^{3:\infty}_{\bar{p},p}(T)\in \mathbb{L}^{3:\infty}_{6,\infty}(T),
   \enq 
   provided that $\frac{2p}{p-1}<3$ for any $p>3$ and standard embedding $\mathbb{L}^{3:\infty}_{\bar{p},p}(T)\hookrightarrow\mathbb{L}^{3:\infty}_{6,\infty}(T)$.  Hence
   by Lemma \ref{L6}, we obtain 
   \beqq\label{wz}
   v^S\in L^{2}([0,T],L^{6,2}).
   \enqq  
   Thanks to \reff{wzu} and \reff{wz}, applying Lemma \ref{heat-energy},
     we obtain 
   \beq
    B(v^S,v^S+2U_f)\in E(T).
    \enq 
    Therefore we obtain $v\in E(T)$
Theorem \ref{NSexistence} is proved.
\end{proof}

\subsection{Uniqueness of $(NSf)$}
Although the solutions in Theorem \ref{NSexistence} need not be unique in $L^{\infty}_t(L^{3,\infty})$, the following arguement shows that the gap between two different solutions has infinite energy.

\begin{proof}[Proof of Theorem \ref{uniquness}]
 Let $u_f\in C_w([0,T^*),L^{3,\infty})$ be a solution to $(NSf)$ constructed in Theorem \ref{NSexistence} with initial data $u_0\in L^{3,\infty}\cap\dot{B}^{s_p}_{p,p}$.

 We now prove the first statement in Theorem \ref{uniquness}:
 
 Assume that $\tilde{u}_f\in C_w([0,T],L^{3,\infty})$ for some $T<T^*$ is another solution to $(NSf)$ with initial data $u_0$ and satisfies $w:=\tilde{u}_f-u_f=w_1+w_2$, where 
 \beq
 w_1\in \mathbb{L}^{r:\infty}_p(T)~~\mathrm{and}~~\|w_2\|_{L^{\infty}(\RR_+,L^{3,\infty})}<4c_1
 \enq 
 for some $p>3$, $2<r<\frac{2p}{p-3}$
 According to Theorem \ref{NSexistence}, $u_f$ can be decomposed as
 \beq
 u_f=v+U_f,
 \enq 
 where $v\in \mathbb{L}^{r:\infty}_p[T<T^*]$ and $U_f\in C_w(\RR_+,L^{3,\infty})$ with $\|U_f\|_{L^{\infty}(\RR_+,L^{3,\infty})}<2c_1$.
 
 We notice that 
 $w$ satisfies:
 \beq
 w&=&B(w,w)+2B(u_f,w)\\
 &=&B(w_1+w_2,w)+B(2u_f,w)\\
 &=&B(w_1+2v,w)+B(w_2+U_f,w).
 \enq 
 
 On the other hand, we notice that for any $q<3$,
 \beq
 \mathcal{L}^{\infty}([0,T],\dot{B}^{s_q}_{q,\infty})\hookrightarrow L^{\infty}([0,T],L^{3,\infty})
 \enq 
 combining with $w_1,v\in \mathbb{L}^{r:\infty}_p(T)$ and $w\in L^{\infty}([0,T],L^{3,\infty})$, using Proposition \ref{bilinear-besov}, we obtain that, for any $\tau\in[0,T]$
 \beqq\label{L-smooth}
 \begin{split}
 \|B(w_1+2v,w)\|_{L^{\infty}([0,\tau],L^{3,\infty})}\lesssim& \|B(w_1+2v,w)\|_{L^{\infty}([0,\tau],\dot{B}_{\bar{p},p}^{s_{\bar{p}}})}\\
 \leq& K \|w_1+2v\|_{\mathcal{L}^{r}([0,\tau];\dot{B}_{p,p}^{s_p+\frac{2}{r}})}\|w\|_{L^{\infty}([0,\tau],L^{3,\infty})}.
 \end{split}
 \enqq 
 And according to Lemma \ref{weak-L^3-continuous}, we obtain that 
 \beq
 \|B(w_2+U_f,w)\|_{L^{\infty}([0,\tau],L^{3,\infty})}\lesssim\|w_2+U_f\|_{L^{\infty}(\RR_+,L^{3,\infty})}\|w\|_{L^{\infty}([0,\tau],L^{3,\infty})}.
 \enq 
 From the smallness of $w_2$ and $U_f$, which is
 \beq
 \|w_2\|_{L^{\infty}(\RR_+,L^{3,\infty})}+\|U_f\|_{L^{\infty}(\RR_+,L^{3,\infty})}<6c_1,
 \enq 
 we obtain that
 \beqq\label{L-small}
 \|B(w_2+U_f,w)\|_{L^{\infty}([0,\tau],L^{3,\infty})}\leq\|w\|_{L^{\infty}([0,\tau],L^{3,\infty})},
 \enqq 
 provided that $c_1$ is small enough.
  
 By \reff{L-smooth} and \reff{L-small}, we obtain that  for any $\tau\in[0,T]$,
 \beq
 \|w\|_{L^{\infty}([0,\tau],L^{3,\infty})}\leq K\|w_1+2v\|_{\mathcal{L}^{r}([0,\tau];\dot{B}_{p,p}^{s_p+\frac{2}{r}})}\|w\|_{L^{\infty}([0,\tau],L^{3,\infty})}.
  \enq 
  By continuity of the norm of $\mathcal{L}^{r}_t(\dot{B}^{s_p+\frac{2}{r}}_{p,p})$ with respect to time, there exists $N$ real numbers $(T_i)_{1\leq i\leq N}$ such that $T_1=0$ and $T_N=T$, satisfying that 
  \beq
  [0,T]=\bigcup^{N-1}_{i=1}[T_i,T_{i+1}]~\mathrm{and}~\|w_1+2v\|_{\mathcal{L}^{r}([T_i,T_{i+1}];\dot{B}_{p,p}^{s_p+\frac{2}{r}})}\leq \frac{1}{2K},
  \enq 
  for all $i\in\{1,\ldots,N-1\}$. 
  
  Now we prove that $w\equiv0$ on  $[T_i,T_{i+1}]$ for all $i\in\{1,\ldots,N-1\}$ by induction. We first notice that 
  \beq
  \|w\|_{L^{\infty}([0,T_2],L^{3,\infty})}&\leq& K\|w_1+2v\|_{\mathcal{L}^{r}([0,T_2];\dot{B}_{p,p}^{s_p+\frac{2}{r}})}\|w\|_{L^{\infty}([0,T_2],L^{3,\infty})}\\
  &\leq&\frac{1}{2}\|w\|_{L^{\infty}([0,T_2],L^{3,\infty})},
  \enq 
  which implies that 
  \beq
  w\equiv0~~~~\mathrm{on}~~[0,T_2].
  \enq 
  Now we assume that $w\equiv0$ on $[0,T_{k}]$ for some $k\geq 2$. Hence 
  \beq
  \mathbf{1}_{[T_{k},T]}(t)w=w=B(w_1+2v,\mathbf{1}_{[T_{k},T]}(t)w)+B(w_2+U_f,\mathbf{1}_{[T_{k},T]}(t)w).
  \enq 
  Therefore we have the following bounds for $w$,
  \beq
  &&\|w\|_{L^{\infty}([T_k,T_{k+1}],L^{3,\infty})}=\|w\|_{L^{\infty}([0,T_{k+1}],L^{3,\infty})}\\
  &\leq&\frac{1}{2} \|B(w_1+2v,w)\|_{L^{\infty}([0,T_{k+1}],L^{3,\infty})}+\frac{1}{2} \|B(w_2+U_f,w)\|_{L^{\infty}([0,T_{k+1}],L^{3,\infty})}.
  \enq 
  Combining with \reff{L-small}, we have
  \beqq\label{k-step}
  \|w\|_{L^{\infty}([T_k,T_{k+1}],L^{3,\infty})}\leq  \|B(\mathbf{1}_{[T_{k},T]}(w_1+2v),w)\|_{L^{\infty}([0,T_{k+1}],L^{3,\infty})}.
  \enqq
  On the other hand, we notice that 
  \beq
  B(w_1+2v,w)=B(w_1+2v,\mathbf{1}_{[T_{k},T]}w)
  =B(\mathbf{1}_{[T_{k},T]}(w_1+2v),w),
  \enq 
  again by Lemma \ref{bilinear-besov}, we obtain that 
  \beq
  &&\|B(w_1+2v,w)\|_{L^{\infty}([0,T_{k+1}],L^{3,\infty})}=\|B(\mathbf{1}_{[T_{k},T]}(w_1+2v),w)\|_{L^{\infty}([0,T_{k+1}],L^{3,\infty})}\\
  &\leq& K \|\mathbf{1}_{[T_{k},T]}(w_1+2v)\|_{\mathcal{L}^{r}([0,T_{k+1}];\dot{B}_{p,p}^{s_p+\frac{2}{r}})}\|w\|_{L^{\infty}([0,T_{k+1}],L^{3,\infty})}\\
  &=&K \|w_1+2v\|_{\mathcal{L}^{r}([T_k,T_{k+1}];\dot{B}_{p,p}^{s_p+\frac{2}{r}})}\|w\|_{L^{\infty}([T_k,T_{k+1}],L^{3,\infty})}\\
  &\leq&\frac{1}{2}\|w\|_{L^{\infty}([T_k,T_{k+1}],L^{3,\infty})}.
  \enq 
  Hence, by the above estimate and \reff{k-step}, we have
  \beq
  \|w\|_{L^{\infty}([T_k,T_{k+1}],L^{3,\infty})}\leq \frac{1}{2}\|w\|_{L^{\infty}([T_k,T_{k+1}],L^{3,\infty})},
  \enq 
  which implies that 
  \beq
  w\equiv0~~\mathrm{on}~~[T_k,T_{k+1}].
  \enq 
  Then we have $w\equiv0$ on $[0,T]$. The first statement in Theorem \ref{uniquness} is proved.
  
  Now we turn to prove the second statement in Theorem \ref{uniquness}:

Assume that $\bar{u}_f\in C_w([0,T],L^{3,\infty})$ for some $T<T^*$ is another solution to $(NSf)$ with  same initial data $u_0$. We denote  $w:=\bar{u}_f-u_f$. By the assumption of the theorem, $w :=\bar{u}_f-u_f\in C([0,T],L^{3,\infty})$ with $w(0)=0$. We notice that $w$ satisfies the following equation on $[0,T]$
   \beq
	w(t)=B(w+2U_f,w)+B(2v,w),
	\enq 
	where $v:=u_f-U_f\in\mathbb{L}^{r_0:\infty}_{p}[T<T^*]$.
   According to Lemma \ref{weak-L^3-continuous}, we have
    \beqq 
	\begin{split}
	\|B(w+&2U_f,w)\|_{L^{\infty}([0,t],L^{3,\infty})}\\
	\leq& C\|w\|^2_{L^{\infty}([0,t],L^{3,\infty})}
	+C\|w\|_{L^{\infty}([0,t],L^{3,\infty})}\|U_f\|_{L^{\infty}([0,t],L^{3,\infty})}\\
	\leq &C\|w\|_{L^{\infty}([0,t],L^{3,\infty})}(\|U_f\|_{L^{\infty}(\RR_+,L^{3,\infty})}+\|w\|_{L^{\infty}([0,t],L^{3,\infty})}).
	\end{split}
	\enqq 

According to the continuity of $w$ in $L^{3,\infty}$ and the fact that $w(0)=0$, one can choose a $T_1$ such that, combined with the smallness of $U_f$,
	\beq
	\|U_f\|_{L^{\infty}([0,t],L^{3,\infty})}+\|w\|_{L^{\infty}([0,t],L^{3,\infty})}\leq \frac{1}{3C},
	\enq 	
	which implies that 
	\beqq\label{w2u}
	\|B(w+2U_f,w)\|_{L^{\infty}([0,T_1],L^{3,\infty})}\leq \frac{1}{3}\|w\|_{L^{\infty}([0,T_1],L^{3,\infty})}.
	\enqq 

By Lemma \ref{bilinear-besov}, by a similar argument as the above paragraph, we have
 that for any $t\in[0,T]$
	\beq
	\|B(2v,w)\|_{L^{\infty}([0,t],L^{3,\infty})}\leq C\|v\|_{\mathcal{L}^{r_0}([0,t],\dot{B}_{p,p}^{s_p+\frac{2}{r_0}})}\|w\|_{L^{\infty}([0,t],L^{3,\infty})}.
	\enq 
By continuity of the norm of $\mathcal{L}^{r_0}([0,t],\dot{B}_{p,p}^{s_p+\frac{2}{r_0}})$ with respect to the time, there exists $T_2>0$ such that 
	\beq
	C\|v\|_{\mathcal{L}^{r_0}([0,T_2],\dot{B}_{p,p}^{s_p+\frac{2}{r_0}})}<\frac{1}{3},
	\enq 
	which implies that 
	\beqq\label{wv}
	\|B(2v,w)\|_{L^{\infty}([0,T_2],L^{3,\infty})}\leq\frac{1}{3}\|w\|_{L^{\infty}([0,T_2],L^{3,\infty})}.
	\enqq 
	According to \reff{w2u} and \reff{wv}, taking $T_0=\min\{T_1,T_2\}$, we have
	\beq
	\begin{aligned}
	\|w&\|_{L^{\infty}([0,T_0],L^{3,\infty})}\\
	&\leq \|B(w+2U_f,w)\|_{L^{\infty}([0,T_0],L^{3,\infty})}+\|B(2v,w)\|_{L^{\infty}([0,T_0],L^{3,\infty})}\\
	&\leq \frac{2}{3}\|w\|_{L^{\infty}([0,T_0],L^{3,\infty})},
	\end{aligned}
	\enq 
	which implies $w\equiv0$ on $[0,T_1]$ and, by continuity, $w\equiv0$ on $[0,T]$ too.
	Therefore we proved the second result in the theorem.
	
	Now we are left with the proof of the last statement of the theorem. Since we need to apply Lemma \ref{gp-energy} to obtain a uniform energy bound, we set $3<p<5$ to make sure that  $v\in \mathcal{L}^{p}([0,T],\dot{B}_{p,p}^{s_p+\frac{2}{p}})$ with $s_p+\frac{2}{p}>0$.

Assume that $\bar{u}_f\in C_w([0,T],L^{3,\infty})$ for some $T<T^*$ is another solution to $(NSf)$ with  same initial data $u_0$. We denote $\omega=\bar{u}_f-u_f$. By the assumption of the theorem, $\omega \in L^{\infty}([0,T],L^2)\cap L^{2}([0,T],\dot{H}^1)$ satisfies the following system:
    \beq
    \left\{
    \begin{array}{ll}
    \p_t \omega-\La \omega+\omega\cdot\nabla \omega+u_f\cdot\nabla \omega+\omega\cdot\nabla u_f=-\nabla \pi,\\
     \nabla\cdot \omega=0,\\
     \omega|_{t=0}=0.
     \end{array}
      \right.
     \enq 
    Therefore we have the following energy equation, for any $t\in(0,T)$,
     \beq
     \|\omega(t)\|_{L^2}^2+2\int_0^t\|\nabla \omega(s)\|^2_{L^2}ds=-2\int_0^t\int_{\RR^3}\omega\cdot\nabla u_f\cdot \omega dxds.
     \enq 
     According to Theorem \ref{NSexistence}, $u_f$ can be written as $u_f=U_f+v$, where $U_f:=NSf(0)$ is the solution to $(NSf)$ with initial data $0$ and $v\in \mathbb{L}^{r_0:\infty}_p[T<T^*]$ is the solution to $(PNS_{U_f})$ with initial data $u_0$.
	 Therefore we have that 
	 \beq
	 \begin{aligned}
	 |\int_0^t\int_{\RR^3}\omega&\cdot\nabla u_f\cdot\omega dxds|\\
	 &\leq |\int_0^t\int_{\RR^3}\omega\cdot\nabla (U_f)\cdot\omega dxds|+|\int_0^t\int_{\RR^3}\omega\cdot\nabla v\cdot\omega dxds|.\\
	 \end{aligned}
	 \enq
	 By Young's inequality in Lorentz spaces, the first term on the right can be controlled by:
	 \beq
	 \begin{aligned}
	 	|\int_0^t\int_{\RR^3}\omega&\cdot\nabla (U_f)\cdot\omega dxds|= |\int_0^t\int_{\RR^3}\omega\cdot\nabla \omega\cdot U_fdxds|\\
	 	 \leq&\int_0^t\|\omega(s)\|_{L^{6,2}}\|\nabla\omega(s)\|_{L^2}\|U_f(s)\|_{L^{3,\infty}}ds.
	 \end{aligned}
	 \enq 
	 We observe now that $\dot{H}^1(\RR^3)\hookrightarrow L^{6,2}(\RR^3)$. This embedding follows from the Young inequality for Lorentz spaces after noticing that $(-\La)^{-\frac{1}{2}}$ is a convolution operator with a function bounded by $\frac{c}{|x|^2}$ whcih therefore belongs to $L^{\frac{3}{2},\infty}$. Hence
	 \beq
	 \int_0^t\|\omega(s)\|_{L^{6,2}}\|\nabla\omega(s)\|_{L^2}\|U_f(s)\|_{L^{3,\infty}}ds
	 \leq \|U_f\|_{L^{\infty}(\RR_+,L^{3,\infty})}\int_0^t\|\nabla\omega(s)\|^2_{L^2}ds.
	 \enq 
	 Since $U_f$ is small enough in $L^{\infty}(\RR_+,L^{3,\infty})$, we obtain 
	 \beq
	 |\int_0^t\int_{\RR^3}\omega\cdot\nabla (U_f)\cdot\omega dxds|\leq \frac{1}{3}\int_0^t\|\nabla \omega(s)\|^2_{L^2}ds.
	 \enq 
	 
	 We recall that $v\in\mathbb{L}^{r_0:\infty}_{p,p}(T)$ with $3<p<5$ and one can take $r_0=\frac{2p}{p-1}$. This implies $v\in\mathcal{L}^{p}([0,T],\dot{B}_{p,p}^{s_p+\frac{2}{p}})$ with $\frac{3}{p}+\frac{2}{p}>1$.
	 Applying Lemma \ref{gp-energy}, we obtain
	 \beq
	 |\int_0^t\int_{\RR^3}\omega\cdot\nabla v\cdot\omega dxds|\leq C\int_0^t\|\omega(s)\|_{L^2}^2\|v(s)\|^p_{\dot{B}_{p,p}^{s_4+\frac{2}{p}}}ds+\int_0^t\|\nabla \omega(s)\|^2_{L^2}ds.
	 \enq 
	 Then $w$ satisfies the following energy inequality,
	\beq
	\|\omega(t)\|_{L^2}^2+\int_0^t\|\nabla \omega(s)\|^2_{L^2}ds\leq C\int_0^t\|\omega(s)\|_{L^2}^2\|v(s)\|^p_{\dot{B}_{p,p}^{s_p+\frac{2}{p}}}ds.
	\enq 
	 By Gronwall's inequality and the fact that $w|_{t=0}=0$, we get 
	\beq
	\|\omega(t)\|_{L^2}^2+\int_0^t\|\nabla \omega(s)\|^2_{L^2}ds\leq0.
	\enq 
	Then $\omega\equiv0$ on [0,T], which implies that $u_f\equiv\bar{u}_f$ on $[0,T]$. Hence we have proved the second statement of Theorem \ref{uniquness}.
	
	Theorem \ref{uniquness} is proved.

\end{proof}

\section{Long-time Behavior and Stability of Global Solutions}
Let $f$ be a given external force satisfying the assumption of Theorem \ref{NSexistence}.
We consider a global in time solution $u_f$ to $(NSf)$ constructed in Theorem \ref{NSexistence} with initial data $u_0\in  L^{3,\infty}\cap\dot{B}^{s_p}_{p,p}$. Also we are interested in the stability of this kind of global solutions.

\subsection{Long-time behavior of global solutions}
Now let us start to prove Theorem \ref{long-time-u}. In order to apply a weak-strong argument, we need to use the regularity result in Theorem \ref{Regu} to obtain the local in time part has a local in time finite energy by a similar argument to the proof of the third statement of Theorem \ref{NSexistence}. However, we need to deal with a more complicated drift term than before.
\begin{proof}
    Let $u_0\in L^{3,\infty}\cap \dot{B}^{s_p}_{p,p}$. Suppose that $u_f\in C_w(\RR_+,L^{3,\infty})$ is a solution to $(NSf)$ with initial data $u_0$ such that 
    \beq
    v:=u_f-U_f\in \mathbb{L}^{r_0:\infty}_p[T<\infty],
    \enq  
    where $U_f:= NSf(0)$ and $r_0=\frac{2p}{p-1}$. By the smallness assumption on $f$, we have $U_f\in L^{\infty}(\RR_+,L^{3,\infty})$. Therefore to prove the theorem, we need to prove $v\in L^{\infty}(\RR_+,L^{3,\infty})$. To achieve this goal, we only need to prove $v\in \mathbb{L}^{r_0:\infty}_p(\infty)$. More precisely, if $v\in \mathbb{L}^{r_0:\infty}_p(\infty)$, by choosing $T=\infty$, $w=U_f$ and $\bar{w}=0$, Theorem \ref{Regu} implies $v$ can be written as
    \beq
    v=v^H+v^S,
    \enq 
    where $v^H=H_{N_0}\in \mathbb{L}^{1:\infty}_p(\infty)$  and $v^S=W_{N_0}+Z_{N_0}\in L^{\infty}(\RR_+,L^{3,\infty})$ with $N_0$ being the largest integer such that $3(N_0-1)<p$.  We recall that in the case when $\bar{w}=0$, $v^H=H_{N_0}$ is a sum of a finite number of multilinear operators of order at most $N_0-1$, acting on $e^{t\La }u_0$ only.
    
     Hence according to $u_0\in L^{3,\infty}$, Lemma \ref{weak-L^3-continuous} implies $H_{N_0}\in L^{\infty}(\RR_+,L^{3,\infty})$. Thus $v\in L^{\infty}(\RR_+,L^{3,\infty})$.

    Now we start to prove that $v\in \mathbb{L}^{r_0:\infty}_p(\infty)$:
    
     We  use the method introduced by C.Calderón in \cite{CC} to prove results on weak solutions in $L^p$ spaces, and used in \cite{GP} in the context of 2D Navier-Stokes equations: we split the initial data into two parts,
 $u_0=\omega_0+\bar{v}_0$, where $\omega_0\in L^{3,\infty}\cap \dot{B}^{s_p}_{p,p}\cap L^2$ and $\bar{v}_0\in L^{3,\infty}\cap \dot{B}^{s_p}_{p,p}$ such that 
	\beq
	\|\bar{v}_0\|_{L^{3,\infty}}<\eps(p)<c(p),
	\enq 
   and its associated solution $\bar{v}$ to $(PNS_{U_f})$ satisfies that $$\|\bar{v}\|_{\mathbb{L}^{r_0:\infty}_{p}(\infty)}\leq C(f)\|\bar{v}_0\|_{L^{3,\infty}}.$$
   We define $\omega:=v-\bar{v}$. It is easy to find that $\omega$ satisfies the following system,

    \beq
\left\{
  \begin{array}{ll}
  \p_t \omega-\La \omega+\omega\cdot\nabla \omega+(U_{f}+\bar{v})\cdot\nabla \omega+\omega\cdot\nabla (U_f+\bar{v})=-\nabla \pi,\\
  \nabla\cdot \omega=0,\\
  \omega|_{t=0}=\omega_0.
  \end{array}
\right.
\enq 
Also $\omega$ can be written as the following integral form
    \beq
    \omega=e^{t\La }\omega_0+B(\omega,v+\bar{v}+2U_f).
    \enq 

\textbf{Step 1:}
We first show that for any $T\in(0,\infty)$, $\omega\in E(T)$. Suppose that $T>0$ is fixed. We notice that $e^{t\La }\omega_0\in E(T)$ provided $\omega_0\in L^2$. Applying Theorem \ref{Regu}, by taking $w=U_f$ and $\bar{w}=v$, we obtain that $\omega$ can be written as
\beq
\omega=\omega^H+\omega^S,
\enq 
where $\omega^H\in \mathbb{L}^{1:\infty}_p(\infty)$  and $\omega^S\in \mathbb{L}^{r_0:\infty}_{\tilde{p},p}$ for some $2<\tilde{p}<3$ . Therefore we obtain
 \beq
 \omega^S\in \mathbb{L}^{3:\infty}_{6,\infty}(T),
 \enq 
 provided that $r_0=\frac{2p}{p-1}<3$ for any $p>3$. Hence by Lemma \ref{heat-energy}, we have 
 \beq
 B(\omega^S,v+\bar{v}+2U_f)\in E(T),
 \enq 
 as $v+\bar{v}+2U_f\in L^{\infty}([0,T],L^{3,\infty})$. 
 
 We recall that $\omega^H=H^E_{N_0}$, where  $H^E_{N_0}$ can be written as
 \beq
 H^E_{N_0}=H^E_{N_0-1}+\sum_{M=0}^{N_0-2}B^M_{N_0-1,N_0-1}(\bar{v}^{\otimes M},v_L^{\otimes(N_0-1-M)}),
 \enq 
 where $B^M_{N_0-1,N_0-1}$ are $(N_{0}-1)$-linear operators and $v_L=e^{t\La }\omega_0$. 
 We recall that 
 \beq
 H^E_2=e^{t\La }\omega_0~~\mathrm{and}~~H^E_3=H^E_2+B(e^{t\La }\omega_0,e^{t\La }\omega_0)+B(\bar{v},e^{t\La }\omega_0).
 \enq 
 Therefore by Lemma \ref{heat-energy} and an inductive argument, we obtain that
 \beq
 H^E_{N_0}\in E(T),
 \enq 
 provided that $\omega_0\in L^2$ and $\bar{v}\in\mathbb{L}^{r_0:\infty}_p(T)$. Applying Lemma \ref{heat-energy} again, we have
 \beq
 B(\omega^H,v+\bar{v}+2U_f)\in E(T),
 \enq 
 as $v+\bar{v}+2U_f\in L^{\infty}([0,T],\dot{B}^{s_p}_{p,\infty})$ deduced by Lemma \ref{embedding}.
 Therefore we obtain that for any $T\in (0,\infty)$, $\omega\in E(T)$.
 
 \textbf{Step 2:} In this step we show a global energy estimate for $\omega$. Let us write an energy estimate in $L^2$, starting at some time $t_0\in (0,\infty)$. We get 
\beq
\|\omega(t)\|_{L^2}^2+2\int_{t_0}^t\|\omega (s)\|_{L^2}^2ds=\|\omega(t_0)\|_{L^2}^2-2\int_{t_0}^t\int_{\RR^3}(\omega\cdot\nabla (\bar{v}+U_f)\cdot \omega dxds.
\enq 
We notice that
\beq
|\int_{t_0}^t\int_{\RR^3}(\omega\cdot\nabla U_f)\cdot \omega dxds|
\leq\|U_f\|_{L^{\infty}(\RR_+,L^{3,\infty})}\int_{t_0}^t\|\omega\|_{L^{6,2}}\|\nabla \omega\|_{L^2}.
\enq 
We recall that $\dot{H}^1(\RR^3)\hookrightarrow L^{6,2}(\RR^3)$, which combined with the above relation implies that
\beq
|\int_{t_0}^t\int_{\RR^3}(\omega\cdot\nabla U_f)\cdot \omega dxds|
\leq\|U_f\|_{L^{\infty}(\RR_+,L^{3,\infty})}\int_{t_0}^t\|\nabla \omega(s)\|^2_{L^2}ds
\enq 
Since $\|U_f\|_{L^{\infty}(\RR_+,L^{3,\infty})}\leq 2c_1(p)$ with $c_1(p)$ is small enough, hence we obtain
\beqq\label{Uf-e} 
|\int_{t_0}^t\int_{\RR^3}(\omega\cdot\nabla U_f)\cdot \omega dxds|
\leq\frac{1}{4}
\int_{t_0}^t\|\nabla \omega(s)\|^2_{L^2}ds.
\enqq  
On the other hand, by a similar argument as above, we have that $\bar{v}$ can be written as, 
  \beq
    \bar{v}=\bar{v}^H+\bar{v}^S,
    \enq 
    where $\bar{v}^H\in \mathbb{L}^{1:\infty}_p(\infty)$ and $\bar{v}^S\in\mathbb{L}^{r_0;\infty}_{\tilde{p},p}$ for some $2<\tilde{p}<3$. Hence
    \beq
    \begin{aligned}
    |\int_{t_0}^t&\int_{\RR^3}(\omega\cdot\nabla \bar{v})\cdot \omega dxds|\\
    &\leq |\int_{t_0}^t\int_{\RR^3}(\omega\cdot\nabla \bar{v}^H)\cdot \omega dxds|+|\int_{t_0}^t\int_{\RR^3}(\omega\cdot\nabla (\bar{v}^{S})\cdot \omega dxds|.
    \end{aligned}
    \enq 
     We recall that $\bar{v}^H$ is a sum of a finite number of multilinear operators of order at most $N_0-1$, acting on $e^{t\La }u_0$ only, as $\bar{v}\in \mathbb{L}^{r_0:\infty}_{p}(\infty)$ is the small global solution to $(PNS_{U_f})$, which is the case of $\bar{w}=0$.   Then by Lemma \ref{kato} (for details see \cite{gip}), we obtain that there exists $K$ only depending on $p$,
    \beq
    \sup_{t>0}t^{\frac{1}{2}}\|\bar{v}^H\|_{L^{\infty}}\lesssim\|\bar{v}_0\|_{\dot{B}^{s_p}_{p,p}}\leq K\eps(p)
    \enq 
    Therefore
    \beqq\label{HN2L}
    \begin{aligned}
    |\int_{t_0}^t\int_{\RR^3}(\omega\cdot\nabla \bar{v}^H)\cdot \omega dxds|&\leq \int_{t_0}^t\|\omega(s)\|_{L^2}\|\nabla\omega\|_{L^2}\sqrt{s}\|\bar{v}^H\|_{L^{\infty}}\frac{ds}{\sqrt{s}}\\
    &\leq \frac{1}{4}\int_{t_0}^{t}\|\nabla\omega(s)\|_{L^2}^2ds+K\eps^2\int_{t_0}^t\|\omega(s)\|^2_{L^2}\frac{ds}{\sqrt{s}}.
    \end{aligned}
    \enqq 
  Again by Theorem \ref{Regu},
   we also notice that there exists $K_1$ only depending on $p$
  \beq
  \|\bar{v}^{S}\|_{L^{\infty}(\RR_+,L^{3,\infty})}\lesssim\|\bar{W}_{N_0}\|_{\mathbb{L}^{r_0:\infty}_{\bar{p},p}(\infty)}+\|\bar{Z}_{N_0}\|_{\mathbb{L}^{r_N;\infty}_{p_{N_2}}(\infty)}\lesssim\|\bar{v}_0\|_{\dot{B}^{s_p}_{p,p}}\leq K_1(\eps(p)).
  \enq 
  Hence we obtain 
  \beq
  \begin{aligned}
 |\int_{t_0}^t\int_{\RR^3}&(\omega\cdot\nabla (\bar{v}^S)\cdot \omega dxds|\\
&\leq\|\bar{v}^S\|_{L^{\infty}(\RR_+,L^{3,\infty})}\int_{t_0}^t\|\omega(s)\|_{L^{6,2}}\|\nabla \omega(s)\|_{L^2}ds\\
&\leq \|\bar{v}^S\|_{L^{\infty}(\RR_+,L^{3,\infty})}\int_{t_0}^t\|\nabla \omega(s)\|^2_{L^2}ds\\
&\leq K_1\eps(p)\int_{t_0}^t\|\nabla \omega(s)\|^2_{L^2}ds.
  \end{aligned}
  \enq 
  Since $\eps(p)$ is small enough, we have
  \beqq\label{wz-e}
  |\int_{t_0}^t\int_{\RR^3}(\omega\cdot\nabla (\bar{v}^S)\cdot \omega dxds|\leq \frac{1}{4}\int_{t_0}^t\|\nabla \omega(s)\|^2_{L^2}ds.
  \enqq 
  According to \reff{Uf-e}, \reff{HN2L} and \reff{wz-e}, we have the following energy estimate for $w$,
  \beq
  \|\omega(t)\|_{L^2}^2+\frac{1}{2}\int_{t_0}^t\|\omega (s)\|_{L^2}^2ds\leq\|\omega(t_0)\|_{L^2}^2+K^2\eps^2\int_{t_0}^t\|\omega(s)\|^2_{L^2}\frac{ds}{\sqrt{s}}.
  \enq 
  We now use Gronwall's Lemma, which yields
  \beq
  \|\omega(t)\|_{L^2}^2+\frac{1}{2}\int_{t_0}^t\|\omega (s)\|_{L^2}^2ds\leq \|\omega(t_0)\|_{L^2}^2\Big(\frac{t}{t_0}\Big)^{K^2\eps^2}.
  \enq

Now by Sobolev embedding and interpolation we have
\beq
\int_{t_0}^t\|\omega(s)\|^4_{\dot{B}^{s_p}_{p,p}}ds\lesssim\int_{t_0}^t\|\omega(s)\|^4_{\dot{H}^{\frac{1}{2}}}ds\leq \int_{t_0}^t\|\omega(s)\|^2_{L^2}\|\nabla \omega(s)\|^2_{L^2}ds,
\enq 
which by the above estimate yields 
\beq
(t-t_0)\inf_{s\in[0,t]}\|\omega(s)\|^4_{\dot{B}^{s_p}_{p,p}}\lesssim \|\omega(t_0)\|_{L^2}^4\Big(\frac{t}{t_0}\Big)^{2K^2\eps^2}.
\enq 
Hence we obtain
\beq
\inf_{s\in[0,t]}\|\omega(s)\|_{\dot{B}^{s_p}_{p,p}}\lesssim \|\omega(t_0)\|_{L^2}\Big(\frac{t}{t_0}\Big)^{K^2\eps^2/2}(t-t_0)^{-\frac{1}{4}}.
\enq 
In particular we can write, for all $t\geq t_0+1$,
\beq
\inf_{s\in[0,t]}\|\omega(s)\|_{\dot{B}^{s_p}_{p,p}}\lesssim \|\omega(t_0)\|_{L^2}t^{\frac{K^2\eps^2}{2}-\frac{1}{4}}
\enq 
which can be made arbitrarily small for $\eps(p)\frac{1}{2K}$ and $t$ large enough. It follows that one can find a time $\tau_0$ such that 
\beq
\|v(\tau_0)\|_{\dot{B}^{s_p}_{p,p}}\leq \eta(p).
\enq 
By Theorem \ref{NSexistence}, we have $v\in \mathbb{L}^{r_0:\infty}_p(\infty)$.

   Theorem \ref{long-time-u} is proved.
\end{proof}

\subsection{Stability of global solutions}
We are now in a position to show the stability of an a priori global solution constructed in Theorem \ref{NSexistence}:
     let us prove Theorem \ref{stability-u}.
     \begin{proof}
	Suppose that a divergence free vector field $u_0\in\dot{B}^{s_p}_{p,p}$ generating a global solution $u_f\in \mathbb{L}^{r_0:\infty}_p[T<\infty]+C_w(\RR_+,L^{3,\infty})$ with $r_0=\frac{2p}{p-1}$ such that $v:=u_f-U_f\in \mathbb{L}^{r_0:\infty}_p[T<\infty]$, where $U_f:=NSf(0)$. According to Theorem \ref{long-time-u}, we obtain that actually
	\beq
	v\in \mathbb{L}^{r_0:\infty}_p(\infty).
	\enq

	Now let $\bar{u}_0\in \dot{B}^{s_p}_{p,p}$ be another divergence free vector field. By Theorem \ref{NSexistence}, there exist a $T^*(\bar{u}_0)$ and a solution $\bar{u}_f\in \mathbb{L}^{r_0:\infty}_p[T<T^*(\bar{u}_0)]+C_w(\RR_+,L^{3,\infty})$ such that $\bar{u}_f-U_f\in  \mathbb{L}^{r_0:\infty}_p[T<T^*(\bar{u}_0)]$. We mention that 
	the life span $T^*(\bar{u}_0)$ is priori finite.	
	
	We denote $w:=\bar{u}_f-u_f$, then it is enough to prove that for $\|w|_{t=0}\|_{\dot{B}^{s_p}_{p,p}}$ small enough $w\in\mathbb{L}^{r_0:\infty}_p(\infty)$.\\
	The function $w$ satisfies the following system:
	\beq
   \left\{
  \begin{array}{ll}
  \p_t w-\La w+w\cdot\nabla w+(v+U_f)\cdot\nabla w+w\cdot\nabla(v+ U_f)=-\nabla \pi,\\
  \nabla\cdot w=0,\\
  w|_{t=0}=w_0.
  \end{array}
\right.
\enq
	We deduce from Proposition 4.1 in \cite{gip} and Lemma \ref{Heat} \& \ref{bilinear-besov} that $w$ satisfies the following estimate:
	\beqq
	\begin{split}
	&\sup_{t\in[\alpha,\beta]}\|w(t)\|_{\dot{B}^{s_p}_{p,p}}+\|w\|_{\mathcal{L}^{r_0}([\alpha,\beta],\dot{B}_{p,p}^{s_p+\frac{2}{r_0}})}\\
	\leq& K\|w(\alpha)\|_{\dot{B}^{s_p}_{p,p}}+K\|w\|^2_{\mathcal{L}^{r_0}([\alpha,\beta],\dot{B}_{p,p}^{s_p+\frac{2}{r_0}})}+K\|v\|_{\mathcal{L}^{r_0}([\alpha,\beta],\dot{B}_{p,p}^{s_p+\frac{2}{r_0}})}\|w\|_{\mathcal{L}^{r_0}([\alpha,\beta],\dot{B}_{p,p}^{s_p+\frac{2}{r_0}})}
	\end{split}
	\enqq 
	for some constant $K>1$ and all times $\alpha,\beta\in[0,T]$.
	Then there exists $N$ real numbers $(T_i)_{1\leq i\leq N}$ such that $T_1=0$ and $T_N=\infty$, satisfying 
	\beqq
	\RR_+=\cup_{i=1}^{N_1}[T_i,T_{i+1}]~\mathrm{and}~\|v\|_{\mathcal{L}^{r_0}([T_i,T_{i+1}],\dot{B}_{p,p}^{s_p+\frac{2}{r_0}})}\leq \frac{1}{4K},~\forall i\in \{1,..N-1\}.
	\enqq 
	Suppose that
	\beqq
	\|w_0\|_{\dot{B}^{s_p}_{p,p}}\leq\frac{1}{8KN(2K)^N}.
	\enqq 
	Then there exists a maximal time $T_0\in\RR_+\cup\{\infty\}$ such that 
	\beqq\label{tilde-v}
	\|w\|_{\mathcal{L}^{r_0}([0,T_0],\dot{B}^{s_p}_{p,p})}\leq \frac{1}{4K}.
	\enqq  
	If $T=\infty$ then the theorem is proved. Suppose now that $T_0<\infty$. Then we can find an integer $k\in\{1,..N_1\}$ such that 
	\beq
	T_k\leq T_0<T_{k+1}.
	\enq 
	Then we have
	\beq
	\|w\|_{\mathcal{L}^{r_0}([T_i,T_{i+1}],\dot{B}_{p,p}^{s_p+\frac{2}{r_0}})}\leq 2K\|w(T_i)\|_{\dot{B}^{s_p}_{p,p}}
	\enq 
	which implies that 
	\beq
	\sup_{t\in[T_i,T_{i+1}]}\|w(t)\|_{\dot{B}^{s_p}_{p,p}}\leq 2K\|w(T_i)\|_{\dot{B}^{s_p}_{p,p}}.
	\enq 
	By induction, we have for all $i\in\{1,...,k-1\}$,
	\beq
	\|w(T_i)\|_{\dot{B}^{s_p}_{p,p}}\leq (2K)^{i-1}\|w_0\|_{\dot{B}^{s_p}_{p,p}}.
	\enq 
	We conclude from the above two results that 
	\beq
	\|w\|_{\mathcal{L}^{r_0}([T_i,T_{i+1}],\dot{B}_{p,p}^{s_p+\frac{2}{r_0}})}\leq (2K)^{i}\|w_0\|_{\dot{B}^{s_p}_{p,p}}
	\enq 
	and
	\beq
	\sup_{t\in [T_i,T_{i+1}]} \|w(t)\|_{\dot{B}^{s_p}_{p,p}}\leq (2K)^{i}\|w_0\|_{\dot{B}^{s_p}_{p,p}}.
	\enq
	for all $i\leq k-1$. The same arguments  as above also apply on the interval $[T_k,T_0]$ and yield
	\beq
	\|w\|_{\mathcal{L}^{r_0}([T_k,T_{0}],\dot{B}_{p,p}^{s_p+\frac{2}{r_0}})}\leq (2K)^{k}\|w_0\|_{\dot{B}^{s_p}_{p,p}}
		\enq 
	and
	\beq
	\sup_{t\in [T_k,T_{0}]} \|w(t)\|_{\dot{B}^{s_p}_{p,p}}\leq (2K)^{k}\|w_0\|_{\dot{B}^{s_p}_{p,p}}.
	\enq
	On the other hand,
	\beq
	\|w\|_{\mathcal{L}^{r_0}([0,T_{0}],\dot{B}_{p,p}^{s_p+\frac{2}{r_0}})}&\leq& \sum_{i=1}^{k-1}\|w\|_{\mathcal{L}^{r_0}([T_i,T_{i+1}],\dot{B}_{p,p}^{s_p+\frac{2}{r_0}})}+\|w\|_{\mathcal{L}^{r_0}([T_k,T_{0}],\dot{B}_{p,p}^{s_p+\frac{2}{r_0}})}\\
	&\leq&N(2K)^N\|w_0\|_{\dot{B}^{s_p}_{p,p}}<\frac{1}{4K}.
	\enq 
	Under assumption \reff{tilde-v} this contracdicts the maximality of $T_0$. Then the theorem is proved.

\end{proof}

\section{Regularity via iteration}
Consider the following equation,
\beqq\label{perturbation}
v(t,x)=e^{t\La }v_0+B(v,v)+B(w,v)+B(\bar{w},v),
\enqq 
where $B$ is defined in \reff{Buv} . 
This section is devoted to showing the regularity of the solution to \reff{perturbation} by using an iteration method introduced in \cite{gip, gkp} and we adopt a similar notation in \cite{gkp}.

\begin{theorem}[Regularity]\label{Regu}
    Let $p>3$ and $2<r_0<\frac{2p}{p-3}$. And let $w\in L^{\infty}(\RR_+, L^{3,\infty})$ and $\bar{w}\in\mathbb{L}^{r_0:\infty}_p(\infty)$. Suppose that $v\in\mathbb{L}^{r_0:\infty}_p(T)$  for some $T>0$  satisfies \reff{perturbation}with initial data $v_0\in \dot{B}^{s_p}_{p,p}$.

     Then for any integer $N\geq 2$ such that $3(N-1)<p$, there are $H_N\in\mathbb{L}^{1;\infty}_p(\infty)$, $W_N\in \mathbb{L}^{r_0:\infty}_{\bar{p},p}$ for some $2<\bar{p}<3$ and $Z_N\in\mathbb{L}^{r_N;\infty}_{p_N}(T)$ with  $p_N:=\frac{p}{N}$ and $r_N=\max\{1,\frac{r_0}{N}\}$ such that,
    \beqq\label{regularity}
    v=H_N+W_N+Z_N.
    \enqq 
    In particular, by taking $N_0:=\max\{N\in \mathbb{N}:N\geq 2, 3(N-1)<p\}$, we obtain that
    $v$ can be written as
    \beq
    v=v^H+v^S,
    \enq 
    where $v^H:=H_{N_0}\in \mathbb{L}^{1:\infty}_p(\infty)$ and $v^S:=W_{N_0}+Z_{N_0}\in\mathbb{L}^{r_0:\infty}_{\tilde{p},p} (T)\hookrightarrow L^{\infty}([0,T],L^{3,\infty})$ with $\tilde{p}:=\max\{\bar{p},p_{N_0}\}$.
    
\end{theorem}
The argument leading to a similar result to the above theorem in the case $w=\bar{w}=0$ can be found in \cite{gip} and \cite{gkp} (in turn inspired by \cite{FP}). The idea of proving Theorem \ref{Regu} is nearly the same as the idea in \cite{gip} and \cite{gkp}. However, since in our case we need to handle two kinds of drift terms, the decomposition via  iteration becomes much more complicated than those results. More precisely, there are two main difference with previous reuslts:
\begin{itemize}
	\item the fact that one of the drift terms $w$ does not have decay in time and cannot be approximated by smooth functions limits the decay in time and the regularity of $W_{N}$. That is no matter how many times we iterate, there is at least one term of $W_N$ only belonging to $\mathbb{L}^{r_0:\infty}_{\bar{p},p}(T)$.
	\item Compared with the previous results in the case when $\bar{w}=0$ (for details, see \cite{gip}), we cannot obtain that $H_N$ belongs to Kato's spaces in general.
     \end{itemize}
In the following, we adapt most of the notations in the proof of Lemma 3.3 in \cite{gkp}.
\begin{proof}
    Let $v\in\mathbb{L}^{r_0:\infty}_p(T)$ for some $T>0$  satisfies \reff{perturbation}. We can write $v$ as
    \beqq\label{v-expansion}
    v=v_L+B(v,v)+B(w,v)+B(\bar{w},v),
    \enqq 
    where
    \beq
    v_L:=e^{t\La }v_0.
    \enq
    This gives the desired expansion when $N=2$:
    We note that
    \beq 
    v=H_2+W_2+Z_2,
    \enq 
    where
    \beq
    H_2=v_L,~~W_2=B(w,v)~~\mathrm{and}~~Z_2=B(v,v)+B(\bar{w},v).
    \enq 
    Lemma \ref{classical-heat} implies that $H_2\in\mathbb{L}^{1:\infty}_p(\infty)$. According to  the second and last statement in Proposition \ref{bilinear-besov}, we have
    \beq
    \|B(v,v)\|_{\mathbb{L}^{\frac{r_0}{2}:\infty}_{\frac{p}{2}}(T)}\lesssim\|v\|^2_{\mathbb{L}^{r_0:\infty}_p(T)}
    \enq 
    and
    \beq
    \|B(\bar{v},v)\|_{\mathbb{L}^{\frac{r_0}{2}:\infty}_{\frac{p}{2}}(T)}\lesssim\|v\|_{\mathbb{L}^{r_0:\infty}_p(T)}\|\bar{w}\|_{\mathbb{L}^{r_0:\infty}_p(\infty)},
    \enq 
    which implies $Z_2\in \mathbb{L}^{\frac{r_0}{2}:\infty}_{\frac{p}{2}}(T)$. \\
    Note that the fact that the bilinear term $B(v,v)$ and linear $B(\bar{w},v)$ allow to pass from an $L^{p}$ to an $L^{\frac{p}{2}}$ integrability is a key feature in this proof.

    We recall the embedding property $L^{3,\infty}\hookrightarrow\dot{B}^{s_q}_{q,\infty}$ for any $q>3$. Combining with the above property with the last statement of Proposition \ref{bilinear-besov} by taking $q=\frac{3p}{p-2}$, we obtain that
    \beq
    \|B(w,v)\|_{\mathbb{L}^{r_0:\infty}_{\frac{6p}{2p+1},p}(T)}\lesssim\|w\|_{L^{\infty}(\RR_+,L^{3,\infty})}\|v\|_{\mathbb{L}^{r_0:\infty}_p(T)}.
    \enq 
    Hence $W_2\in \mathbb{L}^{r_0:\infty}_{\bar{p},p}(T)$ with $\bar{p}=\frac{6p}{2p+1}<3$. Therefore we prove Theorem \ref{Regu} in the case $N=2$.
    
    Next we plug the expansion \reff{v-expansion} in to the term $Z_2(v):=B(v,v)+B(\bar{w},v)$, to find
    \beq
    \begin{aligned}
    	u=&v_L+B(w,v)+B(v,v)+B(\bar{w},v)\\
    	=&v_L+B(w,v)+B(\bar{w},v_L+B(w,v)+B(v,v)+B(\bar{w},v))\\
    	&+B(v_L+B(w,v)+B(v,v)+B(\bar{w},v),v_L+B(w,v)+B(v,v)+B(\bar{w},v))\\
    	=&v_L+B(v_L.v_L)+B(\bar{w},v_L)+B(w,v)+B(\bar{w},B(w,v))+2B(v_L,B(w,v))\\
    	&+2B(B(w,v),B(v,v))+2B(B(w,v),B(\bar{w},v))+B(B(w,v),B(w,v))\\
    	&+2B(v_L,B(v,v))+B(\bar{w},B(v,v))+B(\bar{w},B(\bar{w},v))+2B(v_L,B(v,\bar{w}))\\
    	&+2B(B(v,v),B(\bar{w},v))+B(B(v,v),B(v,v))+B(B(\bar{w},v),B(\bar{w},v)).
    \end{aligned}
    \enq 
    This gives the expansion for $N=3$:
    \beq
    \begin{aligned}
    v=&H_3+W_3+Z_3~~ \mathrm{with}~~H_3=H_2+B(v_L,v_L)+B(\bar{w},v_L),\\
    W_3&=B(w,v)+B(\bar{w},B(w,v))+2B(v_L,B(w,v))\\
    	&+2B(B(w,v),B(v,v))+2B(B(w,v),B(\bar{w},v))+B(B(w,v),B(w,v))
     \end{aligned}
    \enq 
    and
    \beq
    \begin{aligned}
    Z_3=&2B(v_L,B(v,v))+B(\bar{w},B(v,v))+B(\bar{w},B(\bar{w},v))+2B(v_L,B(v,\bar{w}))\\
    	&+2B(B(v,v),B(\bar{w},v))+B(B(v,v),B(v,v))+B(B(\bar{w},v),B(\bar{w},v)).
    \end{aligned}
    \enq 
    The first statement of Proposition \ref{bilinear-besov} implies that $H_3\in\mathbb{L}^{1:\infty}_p(\infty)$ and the expected bounds of $Z_3$ follow again from product laws as soon as $\frac{p}{2}>3$. Now we need to check that $W_3\in \mathbb{L}^{r_0:\infty}_{\bar{p},p}(T)$. According to the previous arguments, we have $B(w,v)\in \mathbb{L}^{r_0:\infty}_{\bar{p},p}(T)$. Hence we obtain that 
    \beq
    B(w,v)\in L^{\infty}([0,T],\dot{B}^{s_q}_{q,\infty}),\forall q>3,
    \enq 
    provided that $\bar{p}<3$. Again by the last statement of Proposition \ref{bilinear-besov} and taking $q=\frac{3p}{p-2}$, we have the rest of terms in $Z_3$ belong to $\mathbb{L}^{r_0:\infty}_{\bar{p},p}(T)$,
    which implies that $W_3\in \mathbb{L}^{r_0:\infty}_{\bar{p},p}(T)$.
     
    Iterating further, the formulas immediately get very long and complicated, so let us argue by induction:
    
    Assume that for any $2\leq N\leq N_0$, there is an integer $K_N\geq0$, and for any $0\leq k\leq K_N$ some $(N+k)$-linear operators $B^M_{N+k,N}$(the parameter $M\in\{1,\ldots,N+k\}$ measures the number of entries in which $v$ and $\bar{w}$, rather than $v_L$, appears and the second parameter in the subscript denotes that the operators are generated in $N$th step) such that
    \beq
    v=H_N+W_N+Z_N
    \enq 
    with for any $N\geq 3$
    \beqq\label{HN}
    H_N=H_{N-1}+\sum_{M=0}^{N-2}B^M_{N-1,N-1}(\bar{w}^{\otimes M},v_L^{\otimes(N-1-M)}),
    \enqq
    $Z_N$ may be written as the form
    \beqq
    \begin{split}
    Z_N=&\sum_{M=1}^{N}\sum_{\mbox{\tiny$\begin{array}{c}
J+L=M,\\
J\geq1\end{array}$}}B^M_{N,N}(v^{\otimes J},\bar{w}^{\otimes L},v_L^{\otimes(N-M)})\\
    &+\sum_{k=1}^{K_N}\sum_{M=0}^{N+k}\sum_{J+L=M}B^M_{N+k,N}(v^{\otimes J},\bar{w}^{\otimes L},v_L^{\otimes(N+k-M)}),
    \end{split}
    \enqq 
    and 
    \beqq\label{WN}
    \begin{split}
    W_N=\sum_{M=1}^{N-1}&\sum_{J+L=M}\sum_{l=1}^{J-1}\sum_{i+j=l}\sum_{i+j+m=J}B^M_{{N-1},{N-1}}(B(v,v)^{\otimes i},B(w,v)^{\otimes m},\\
    &B(\bar{w},v)^{\otimes j},\bar{w}^{\otimes L},v_L^{\otimes(N-1-L-l-m)})\\
&+W_{N-1}.
\end{split}
    \enqq
    we have used the following convention: for any $J+L=M$
    \beq
    B^M_{N+k,N}(\underbrace{u,\cdots,u}_{J~\mathrm{terms}},\underbrace{v,\cdots,v}_{L~\mathrm{terms}},\underbrace{w,\cdots,w}_{N-M~\mathrm{terms}}):=B^M_{N+k,N}(u^{\otimes J},v^{\otimes L},w^{\otimes(N-M)})
    \enq 
    
    Now let us prove that for any $2\leq N\leq N_0$
    \beqq\label{N,N+1}
    \begin{split}
        	Z_N=&\sum_{M=1}^{N}\sum_{J+L=M}\sum_{l=0}^{J-1}\sum_{i+j=l}\sum_{i+j+m=J}B^M_{{N},{N}}(B(v,v)^{\otimes i},B(w,v)^{\otimes m},\\
    &B(\bar{w},v)^{\otimes j},\bar{w}^{\otimes L},v_L^{\otimes(N-L-l-m)})\\
&+\sum_{M=0}^{N-1}B^M_{N,N}(\bar{w}^{\otimes M},v_L^{\otimes(N-M)})+
Z_{N+1}
    \end{split}
    \enqq
    where $Z_{N+1}$ can be written in the following way: there exists an integer $K_{N+1}\geq0$ for for all $0\leq k\leq K_{N+1}$ and $0\leq M\leq N+1+k$, some $N+1+k$-linear operators $B^{M}_{N+1+k,N+1}$, such that 
      \beqq\label{N+1}
    \begin{split}
    Z_{N+1}=&\sum_{M=1}^{N+1}\sum_{\mbox{\tiny$\begin{array}{c}
J+L=M,\\
J\geq1\end{array}$}}B^M_{N+1,N+1}(v^{\otimes J},\bar{w}^{\otimes L},v_L^{\otimes(N+1-M)})\\
    &+\sum_{k=1}^{K_{N+1}}\sum_{M=0}^{N+k+1}\sum_{J+L=M}B^M_{N+k+1,N+1}(v^{\otimes J},\bar{w}^{\otimes L},v_L^{N+1+k-M}).
    \end{split}
    \enqq  
    
    In order to prove \reff{N,N+1} and \reff{N+1} we just need to use \reff{v-expansion} again: replacing $v$ by $v_L+B(v,v)+B(w,v)+B(\bar{w},v)$ in the argument of $B^M_{N,N}$ gives
    \beq
    \begin{aligned}
    B^M_{N,N}&(v^{\otimes J},\bar{w}^{\otimes L},v_L^{\otimes(N-M)})\\
    =&B^M_{N,N}((v_L+B(w,v)+B(v,v)+B(\bar{w},v))^{\otimes J},\bar{w}^{\otimes L},v_L^{\otimes(N-M)})\\
    =&B^M_{N,N}(v_L^{\otimes J},\bar{w}^{\otimes L},v_L^{\otimes(N-M)})\\
    &+\sum_{l=0}^{J-1}\sum_{i+j=l}\sum_{i+j+m=J}B^M_{{N},{N}}(B(v,v)^{\otimes i},B(w,v)^{\otimes m},B(\bar{w},v)^{\otimes j},\bar{w}^{\otimes L},v_L^{\otimes(N-L-l-m)})\\
    &+\sum_{l=1}^{J}\sum_{i+j=l}\tilde{B}^M_{{N+l},{N}}(v^{\otimes (2i+j)},\bar{w}^{\otimes (L+j)},v_L^{\otimes(N-L-l)}),
    \end{aligned}
    \enq 
    where $\tilde{B}^M_{{N+l},{N}}$ are some $N+l$-linear operators. Therefore we have 
    \beq
    \begin{aligned}
    Z_{N+1}=&\sum_{k=1}^{K_{N}}\sum_{M=0}^{N+k}\sum_{J+L=M}B^M_{N+k,N}(v^{\otimes J},\bar{w}^{\otimes L},v_L^{\otimes(N+k-M)})\\
    &+\sum_{M=1}^{N}\sum_{J+L=M}\sum_{l=1}^{J}\sum_{i+j=l}\tilde{B}^M_{{N+l},{N}}(v^{\otimes (2i+j)},\bar{w}^{\otimes (L+j)},v_L^{\otimes(N-L-l)})
    \end{aligned}
    \enq 
    after reordering, this proves \reff{N,N+1} and \reff{N+1}. Moreover \reff{N,N+1} and \reff{N+1} imply that \reff{HN} and \reff{WN} hold for the case that $N=N_0+1$. 
    
    To conclude the proof the theorem it remains to prove that $H_N\in\mathbb{L}^{1;\infty}_p(\infty)$, $W_N\in \mathbb{L}^{r_0:\infty}_{\bar{p},p}$ for some $2<\bar{p}<3$ and $Z_N\in\mathbb{L}^{r_N;\infty}_{p_N}(T)$ with  $p_N:=\frac{p}{N}$ and $r_N=\max\{1,\frac{r_0}{N}\}$. In fact, the above results again follow from estimates about the heat flow (see Lemma \ref{classical-heat}) and product laws in Proposition \ref{bilinear-besov}, which are based on a similar argument of the cases that $N=2,N=3$.  
       
   Now we take $N_0:=\max\{N\in \mathbb{N}:N\geq 2, 3(N-1)<p\}$. It is obvious that $p_{N_0}=\frac{p}{N_0}<3$, which implies that 
   \beq
   v^S=W_{N_0}+Z_{N_0}\in \mathbb{L}^{r_0:\infty}_{\tilde{p},p}(T)\hookrightarrow L^{\infty}([0,T],L^{3,\infty})
   \enq  
   provided Lemma \ref{embedding}.
   
   Theorem \ref{Regu} is proved.

\end{proof}

\section{Appendix}

\subsection{Estimates on the heat equation}
For the completeness of our proof, we give standard estimates for the heat kernel in Besov space. A similar result can be found in \cite{C}. We first recall the long-time behavior of heat flow. We mention that the following lemmas only focus on critical Besov spaces.
\begin{lemma}\label{classical-heat}
	Let $p,q\in [1,\infty)$ and $g\in\dot{B}^{s_p}_{p,q}$. Then we have that
	\beq
	e^{t\La }g\in\mathbb{L}^{1:\infty}_{p,q}(\infty)
	\enq 
	and
	\beq
	\lim_{t\to\infty}\|e^{t\La }g\|_{\dot{B}^{s_p}_{p,q}}=0.
	\enq 
\end{lemma}
\begin{proof}
	Let $g\in\dot{B}^{s_p}_{p,p}$. We notice that 
	for any $j\in\mathbb{Z}$,
	\beq
	\|e^{t\La }\La_j g\|_{L^p}\lesssim e^{-t2^{2j}}\|\La_j g\|_{L^p}\lesssim 2^{-js_p}e^{-t2^{2j}}c_{j,q}\|g\|_{\dot{B}^{s_p}_{p,q}},
	\enq 
	where $\|(c_{j,q})_{j\in\mathbb{Z}}\|_{\ell^q}=1$.
    Then for any $r\in[1,\infty]$, we have
    \beq
    \|e^{t\La }\La_j g\|_{L^r(\RR_+,L^p_x)}\lesssim 2^{-js_p}2^{-\frac{2j}{r}}c_{j,q}\|g\|_{\dot{B}^{s_p}_{p,q}},
    \enq 
    which  implies that 
    \beq
    \big\|\big(2^{j(s_p+\frac{2}{r})}\|e^{t\La }\La_j g\|_{L^r(\RR_+,L^p_x)}\big)_{j\in\mathbb{Z}}\big\|_{\ell^q}\lesssim\|g\|_{\dot{B}^{s_p}_{p,q}}
    \enq 
    Hence we have $e^{t\La }g\in\mathbb{L}^{1:\infty}_{p,q}(\infty)$.

	Moreover for any $\eps>0$, one can choose an integer $N$ such that for any $t\geq 0$
	\beq
	\Big(\sum_{|j|>N}2^{qjs_p}\|e^{t\La}\La_j g\|^{q}_{L^p}\Big)^{\frac{1}{q}}<\frac{\eps}{2}.
	\enq 
	Also we have
	\beq
	\sum_{|j|\leq N}2^{qjs_p}\|e^{t\La}\La_j g\|^{q}_{L^p}\lesssim 2^{-jqs_p}e^{-qt2^{-2N}}c^q_{j,q}\|g\|^q_{\dot{B}^{s_p}_{p,q}},
	\enq 
	hence for the fixed $N$, there exists a $T(\eps)>0$ such that for any $t>T$,
	\beq
	\Big(\sum_{|j|\leq N}2^{qjs_p}\|e^{t\La}\La_j g\|^{q}_{L^p}\Big)^{\frac{1}{q}}<\frac{\eps}{2}.
	\enq 
	Therefore we have that for any $\eps>0$, there exists a $T(\eps)>0$, such that for any $t>T$
	\beq
	\|e^{t\La}g \|_{\dot{B}^{s_p}_{p,q}}<\eps.
	\enq 
	The lemma is proved.
\end{proof}
\begin{lemma}\label{Heat}
	Let $p\in[1,\infty]$ and $r\in[1,\infty]$. Suppose that $f$ is a function belonging to $\mathcal{L}^r_T(\dot{B}^{s_p+\frac{2}{r}-2}_{p,p})$. We denote that, for any $t\in[0,T]$
	\beqq\label{h(f)}
	H(f):=\int_0^t e^{(t-s)\La}f(s,\cdot)ds.
	\enqq
	Then we have $H(f)\in \mathcal{L}^{\bar{r}}_T(\dot{B}^{s_p+\frac{2}{\bar{r}}}_{p,p})$ for any $\bar{r}\geq r$, and	
	\beq
	\|H(f)\|_{\mathcal{L}^{\bar{r}}_T(\dot{B}^{s_p+\frac{2}{\bar{r}}}_{p,p})}\lesssim\|f\|_{\mathcal{L}^r_T(\dot{B}^{s_p+\frac{2}{r}-2}_{p,p})}.
	\enq 
    Moreover, if $r<\infty$,
    \beq
	\lim_{t\to\infty}\|H(f)\|_{\dot{B}^{s_p}_{p,p}}=0.
	\enq
\end{lemma}
\begin{proof}
	We first notice that 
	\beq
	\begin{aligned}
	\|\La_jH(f)\|_{L^{\bar{r}}_tL^p_x}&\leq \|\int_0^t\|e^{(t-s)\La }\La_j f(s,\cdot)\|_{L^p_x}ds\|_{L^{\bar{r}}_t}\\
	&\lesssim\|\int_0^t e^{-c(t-s)2^{2j}}\|\La_j f(s,\cdot)\|_{L^p_x}ds\|_{L^{\bar{r}}_t}\\
	&\lesssim\|e^{-ct2^{2j}}\|_{L^{r'}_t}\|\La_j f\|_{L^r_tL^p_x},
	\end{aligned}
	\enq 
	where $\frac{1}{\bar{r}}+1=\frac{1}{r}+\frac{1}{r'}$.
	Since $f\in \mathcal{L}^r_T(\dot{B}^{s_p+\frac{2}{r}-2}_{p,p})$ we have 
	\beq
	\|\La_j f\|_{L^r_tL^p_x}\lesssim 2^{-j(s_p+\frac{2}{r}-2)}d_{j,p}\|f\|_{\mathcal{L}^r_T(\dot{B}^{s_p+\frac{2}{r}-2}_{p,p})},
	\enq 
	where $(d_{j,p})\in\ell^p$ and $\|(d_{j,p})\|_{\ell^p}=1$. We also notice that 
	\beq
	\|e^{-ct2^{2j}}\|_{L^{r'}_t}\lesssim 2^{-\frac{2j}{r'}}.
	\enq
	Then we have 
	\beq
	\|\La_jH(f)\|_{L^{\bar{r}}_tL^p_x}&\lesssim&2^{-j(s_p+\frac{2}{r}+\frac{2}{r'}-2)}d_{j,p}\|f\|_{\mathcal{L}^r_T(\dot{B}^{s_p+\frac{2}{r}-2}_{p,p})}\\
	&=&2^{-j(s_p+\frac{2}{\bar{r}})}d_{j,p}\|f\|_{\mathcal{L}^r_T(\dot{B}^{s_p+\frac{2}{r}-2}_{p,p})},
	\enq 
	which implies that 
	\beq
	\Big\|\big(2^{j(s_p+\frac{2}{\bar{r}})}\|\La_jH(f)\|_{L^{\bar{r}}_tL^p_x}\big)_{j\in\mathbb{Z}}\Big\|_{\ell^p}\lesssim\|f\|_{\mathcal{L}^r_T(\dot{B}^{s_p+\frac{2}{r}-2}_{p,p})}.
	\enq 
    Thus we proved that $H(f)\in \mathcal{L}^{\bar{r}}_T(\dot{B}^{s_p+\frac{2}{\bar{r}}}_{p,p})$ for any $\bar{r}\geq r$, and	
	\beq
	\|H(f)\|_{\mathcal{L}^{\bar{r}}_T(\dot{B}^{s_p+\frac{2}{\bar{r}}}_{p,p})}\lesssim\|f\|_{\mathcal{L}^r_T(\dot{B}^{s_p+\frac{2}{r}-2}_{p,p})}.
	\enq 

    Now we suppose that $r<\infty$.\\
	First we decompose $H(f)$ into two parts:
	\beq
	H_1(f):=\int_0^\frac{t}{2}e^{(t-s)\La}f(s,\cdot)ds,
	\enq 
	and
	\beq
	H_2(f):=\int^t_\frac{t}{2}e^{(t-s)\La}f(s,\cdot)ds.
	\enq 
	We notice that $H_1(f)$ can be written as 
	\beq
	H_1(f)=e^{\frac{t\La}{2}}\int_0^{\frac{t}{2}}e^{(\frac{t}{2}-s)\La}f(s,\cdot)ds=e^{\frac{t\La }{2}}H(f)(\frac{t}{2}).
	\enq 
	According the above argument, we have, for any $t>0$, $H(f)(\frac{t}{2})\in\dot{B}^{s_p}_{p,p}$. Applying Lemma \ref{classical-heat}, we have
	\beq
	\lim_{t\to\infty}\|e^{\frac{t\La }{2}}H(f)(\frac{t}{2})\|_{\dot{B}^{s_p}_{p,p}}=0.
	\enq 
	Now we turn to $H_2(f)$, we have
	\beq
	\|\La_j H_2(f)\|_{L^p}&\lesssim&\int_{\frac{t}{2}}^t e^{-2^{2j}(t-s)}\|\La_j f(s)\|_{L^p}ds\\
	&\lesssim&2^{2j(\frac{1}{r}-1)}\|\La_j f\|_{L^r([t/2,\infty);L^p)},
	\enq 
	which implies that 
	\beq
	\|H_2(f)(t)\|_{\dot{B}^{s_p}_{p,p}}\lesssim\|f\|_{\mathcal{L}^r([t/2,\infty);\dot{B}^{s_p+\frac{r}{2}-2}_{p,p})}\to0,~~\mathrm{as}~~t\to\infty.
	\enq 
	Lemma \ref{Heat} is proved.
\end{proof}
\subsection{Product laws in Besov spaces}
In this paragraph we recall the following product laws in Besov spaces, which use the theory of paraproducts. We only elected to state the results we needed previously, but it should be clear that we have not stated all possible estimates in their greatest generality.
\begin{proposition}\label{bilinear-besov}

	\begin{enumerate}
		\item Let $p>3$ and $2<r<\frac{2p}{p-3}$. Then there exists a constant $\gamma>0$ such that for any
		 $v,w\in \mathcal{L}^{r}([0,T],\dot{B}^{s_p+\frac{2}{r}}_{p,p})$, we have  
	\beqq
	&\|vw\|_{\mathcal{L}^{r}([0,T],\dot{B}^{s_p+\frac{2}{r}-1}_{p,p})}\leq \gamma\|v\|_{\mathcal{L}^{r}([0,T],\dot{B}^{s_p+\frac{2}{r}}_{p,p})}\|w\|_{\mathcal{L}^{r}([0,T],\dot{B}^{s_p+\frac{2}{r}}_{p,p})}.
	\enqq 
	\item Let $p_1,p_2\in (3,\infty)$, $2<r<\frac{2p}{p-3}$ and $T\in\RR_+\cup\{\infty\}$. Suppose that $v\in\mathbb{L}^{r;\infty}_{p_1}(T)$ and $w\in\mathbb{L}^{r;\infty}_{p_2}(T)$. Then we have 
	\beq
	\|vw\|_{\mathcal{L}^{\frac{r}{2}}([0,T],\dot{B}_{p,p}^{s_p+\frac{4}{r}-1})}\lesssim\|v\|_{\mathbb{L}^{r;\infty}_{p_1}(T)}\|w\|_{\mathbb{L}^{r;\infty}_{p_2}(T)},
	\enq
	where $\frac{1}{p}=\frac{1}{p_1}+\frac{1}{p_2}$.
     
	\item Let $p>3$.  Suppose that $w\in L^{\infty}([0,T],L^{3,\infty})$ and $v\in\mathcal{L}^{r_0}([0,T];\dot{B}_{p,p}^{s_p+\frac{2}{r_0}})$ for some $T\in\RR_+\cup\{+\infty\}$ with $r_0=\frac{2p}{p-1}$, then we have 
	\beq
	\|vw\|_{\mathcal{L}^{r_0}([0,T],\dot{B}^{s_{\bar{p}}+\frac{2}{r_0}-1}_{\bar{p},p})}\leq C(p)\|w\|_{L^{\infty}([0,T],L^{3,\infty})}\|v\|_{\mathcal{L}^{r_0}([0,T];\dot{B}_{p,p}^{s_p+\frac{2}{r_0}})},
	\enq 
	where $\frac{1}{\bar{p}}=\frac{1}{3}+\frac{1}{6p}$ and $C(p)\to\infty$ as $p\to\infty$.

	\end{enumerate}
\end{proposition}
Since the first two results in the proposition are standard and well-known, which can be found in \cite{C,gip}, we only give the proof of the last of the proposition. 

\begin{proof}
    For simplicity, we treat $w$ and $v$ as functions. We have 
	\beq
	\La_j wv=\La_j T_w v+\La_j T_v w+\La_j R(u,v).
	\enq 
	We first take $q_1$ such that $\frac{1}{\bar{p}}=\frac{1}{p}+\frac{1}{q_1}=\frac{1}{3}+\frac{1}{6p}$ implying that $q_1=\frac{6p}{2p-5}>3$. 
	
		About $\La_j T_w v$, we have
	\beq
	\|\La_j T_w v\|_{L^{r_0}(L^{\bar{p}})}\lesssim\|(S_j w)(\La_j v)\|_{L^{r_0}(L^{\bar{p}})}\lesssim\|S_j w\|_{L^\infty(L^{q_1})}\|v\|_{L^{r_0}(L^p)}.
	\enq 
	And we notice that
	\beq
    \|S_j w\|_{L^\infty(L^{q_1})}\lesssim\sum_{j'\leq j}\|\La_j w\|_{L^\infty(L^{q_1})}\lesssim \sum_{j'\leq j}2^{-j's_{q_1}}c_{j,\infty}\|w\|_{\mathcal{L}^{\infty}([0,T],\dot{B}_{q_1,\infty}^{s_{q_1}})},
	\enq 
	and
	\beq
	\|v\|_{L^{r_0}(L^p)}\lesssim 2^{-j(s_p+\frac{2}{r_0})}c_{j,p'}\|v\|_{\mathcal{L}^{r_0}([0,T];\dot{B}_{p,p}^{s_p+\frac{2}{r_0}})}.
	\enq 
	Since $s_{q_1}<0$, we have
	\beq
	\|2^{j(s_{\bar{p}}+\frac{2}{r_0}-1)}\|\La_j T_w v\|_{L^{r_0}(L^{p'})}\|_{\ell^{p'}}\lesssim \|w\|_{\mathcal{L}^{\infty}([0,T],\dot{B}_{q_1,\infty}^{s_{q_1}})}\|v\|_{\mathcal{L}^{r_0}([0,T];\dot{B}_{p,p}^{s_p+\frac{2}{r_0}})}.
	\enq  
	This combined with Lemma \ref{embedding}, implies that
	\beqq\label{twv}
	\|2^{j(s_{\bar{p}}+\frac{2}{r_0}-1)}\|\La_j T_w v\|_{L^{r_0}(L^{p'})}\|_{\ell^{p'}}\lesssim \|w\|_{L^{\infty}([0,T],L^{3,\infty})}\|v\|_{\mathcal{L}^{r_0}([0,T];\dot{B}_{p,p}^{s_p+\frac{2}{r_0}})}.
	\enqq 
	
	Now we choose  $q:=\frac{12p}{4p-1}$ and $p_1:=4p$. It is easy to check such that $\frac{1}{\bar{p}}=\frac{1}{p_1}+\frac{1}{q}=\frac{1}{3}+\frac{1}{6p}$ . We notice that 
	\beq
	\|\La_j T_v w\|_{L^{r_0}(L^{\bar{p}})}\lesssim\|S_j v\|_{L^{r_0}(L^{p_1})}\|\La_j w\|_{L^{\infty}(L^q)},
	\enq 
	and
	\beq
	\|S_j v\|_{L^{r_0}(L^{p_1})}&\lesssim&\sum_{j'\leq j}\|\La_j v\|_{L^{r_0}(L^{p_1})}\lesssim 2^{-j'(s_{p_1}+\frac{2}{r_0})}c_{j',p}\|v\|_{\mathcal{L}^{r_0}([0,T];\dot{B}_{p_1,p}^{s_{p_1}+\frac{2}{r_0}})}\\
	&\lesssim&2^{-j'(s_{p_1}+\frac{2}{r_0})}c_{j',p}\|v\|_{\mathcal{L}^{r_0}([0,T];\dot{B}_{p,p}^{s_{p}+\frac{2}{r_0}})},
	\enq 
	and 
	\beq
	\|\La_j w\|_{L^{\infty}(L^q)}\lesssim 2^{-js_q}c_{j,\infty}\|w\|_{\mathcal{L}^{\infty}([0,T],\dot{B}_{q,\infty}^{s_{q}})}.
	\enq 
	Since  $s_{p_1}+\frac{2}{r_0}=-1+\frac{3}{4p}+1-\frac{1}{p}=-\frac{1}{4p} <0$,  we have 
	\beqq\label{tvw}
	\|2^{j(s_{\bar{p}}+\frac{2}{r_0}-1)}\|\La_j T_w v\|_{L^{r_0}(L^{\bar{p}})}\|_{\ell^{p}}\lesssim p \|w\|_{\mathcal{L}^{\infty}([0,T],\dot{B}_{q,\infty}^{s_{q}})}\|v\|_{\mathcal{L}^{r_0}([0,T];\dot{B}_{p,p}^{s_{p}+\frac{2}{r_0}})}.
	\enqq 
	Again by Lemma \ref{embedding}, we have
	\beqq\label{tvw}
	\|2^{j(s_{\bar{p}}+\frac{2}{r_0}-1)}\|\La_j T_w v\|_{L^{r_0}(L^{\bar{p}})}\|_{\ell^{p}}\lesssim p \|w\|_{L^{\infty}([0,T],L^{3,\infty})}\|v\|_{\mathcal{L}^{r_0}([0,T];\dot{B}_{p,p}^{s_{p}+\frac{2}{r_0}})}.
	\enqq 
	Now we turn to the remainder $\La_j R(w,v)$. We denote that $\frac{1}{\tilde{p}}:=\frac{1}{p}+\frac{1}{q}=\frac{1}{3}+\frac{11}{12p}$.
	Since 
	\beq
	\begin{aligned}
	\|\La_j R(w,v)&\|_{L^{r_0}(L^{\tilde{p}})}\lesssim\sum_{k\geq j-1}\|\La_k w\|_{L^{\infty}(L^q)}\|\La_k v\|_{L^{r_0}(L^p)}\\
	&\lesssim\sum_{k\geq j}2^{-k(s_q+s_p+\frac{2}{r_0})}c_{k,\infty}c_{k,p}\|w\|_{\mathcal{L}^{\infty}([0,T],\dot{B}_{q,\infty}^{s_{q}})}\|v\|_{\mathcal{L}^{r_0}([0,T];\dot{B}_{p,p}^{s_{p}+\frac{2}{r_0}})},
	\end{aligned}
	\enq 
	and  
	\beq
	s_p+s_q+\frac{2}{r_0}=\frac{7}{4p}>0
	\enq 
	 we have that, by applying Lemma \ref{embedding},
	\beq
	\begin{aligned}
	\|2^{j(s_{\tilde{p}}+\frac{2}{r_0}-1)}\|\La_j R(w,v)\|_{L^{r_0}(L^{\tilde{p}})}\|_{\ell^{p}}\lesssim&p\|w\|_{\mathcal{L}^{\infty}([0,T],\dot{B}_{q,\infty}^{s_{q}})}\|v\|_{\mathcal{L}^{r_0}([0,T];\dot{B}_{p,p}^{s_{p}+\frac{2}{r_0}})}\\
	\lesssim &p\|w\|_{L^{\infty}([0,T],L^{3,\infty})}\|v\|_{\mathcal{L}^{r_0}([0,T];\dot{B}_{p,p}^{s_{p}+\frac{2}{r_0}})}
	\end{aligned}
	\enq 

	which is $R(w,v)\in\mathcal{L}^{r_0}([0,T];\dot{B}_{\tilde{p},p'}^{s_{\tilde{p}}+\frac{2}{r_0}-1})$.\\
	And we have $R(w,v)\in\mathcal{L}^{r_0}([0,T];\dot{B}_{\bar{p},p}^{s_{\bar{p}}+\frac{2}{r_0}-1})$, as $\tilde{p}<\bar{p}$. Combining with \reff{twv} and \reff{tvw} we get
	\beq
	\|w v\|_{\mathcal{L}^{r_0}([0,T];\dot{B}_{\bar{p},p}^{s_{\bar{p}}+\frac{2}{r_0}-1})}\leq C(p) \|w\|_{L^{\infty}([0,T],L^{3,\infty})}\|v\|_{\mathcal{L}^{r_0}([0,T];\dot{B}_{p,p}^{s_{p}+\frac{2}{r_0}})},
	\enq 
	where $C(p)\to\infty$ as $p\to\infty$.
	The proposition is proved.
\end{proof}
We also recall the following standard embedding without proof. For details of the proof, one can check \cite{BL,FP}.
\begin{lemma}\label{embedding}
	Let $q_1<3<q_2$. Then the following embeddings hold:
	\beq
	\dot{B}^{s_{q_1}}_{q_1,\infty}\hookrightarrow L^{3,\infty}\hookrightarrow \dot{B}^{s_{q_2}}_{q_1,\infty}.
	\enq 
\end{lemma}

\subsection{Properties of the bilinear operator $B$}
We show a  well-known results on the continuity of $B(u,v)$ in Kato's space by using the spatial decay of the convolution kernel appearing in $B$ (see \cite{FP})
\begin{lemma}\label{kato}
 Let $p>3$.Suppose that $u,v\in K_p(\RR^3)$, then 
		      \beqq
		      \|B(u,v)\|_{K_p}\lesssim\|u\|_{K_p}\|v\|_{K_p}.
		      \enqq  
		Moreover if  $p>6$, then
		\beqq
		\|B(u,v)\|_{K_{\infty}}\leq \|u\|_{K_p}\|v\|_{K_p}.
		\enqq 
\end{lemma}

And we recall that $B$ is a bounded operator from $L^{\infty}([0,T],L^{3,\infty})\times L^{\infty}([0,T],L^{3,\infty})$ to $L^{\infty}([0,T],L^{3,\infty})$ for any $T\in \RR_+\cup\{+\infty\}$ (see \cite{bbis})
\begin{lemma}\label{weak-L^3-continuous}
    Suppose that $u,v\in L^{\infty}([0,T],L^{3,\infty})$ for some $T\in \RR_+\cup\{+\infty\}$. Then 
	\beq
	\|B(u,v)\|_{L^{\infty}([0,T],L^{3,\infty})}\lesssim\|u\|_{L^{\infty}([0,T],L^{3,\infty})}\|u\|_{L^{\infty}([0,T],L^{3,\infty})}.
	\enq 
	Moreover, $B(u,v)\in C_w([0,T],L^{3,\infty})$.

\end{lemma}
The following lemma is a particular case of the result about the continuity of the trilinear form $\int_0^T\int_{\RR^3}(a\cdot\nabla b)\cdot cdxdt$  proved by I.Gallagher \& F. Planchon in \cite{GP}.
\begin{lemma}\label{gp-energy}
	Let $d\geq 2$ be fixed, and let $r$ and $q$ be two real numbers such that $2\leq 2<\infty, 2<q<+\infty$. Suppose $a\in L^\infty(\RR_+,L^2)\cap L^2(\RR_+,\dot{H}^1)$ and $c\in \mathcal{L}^q([0,T],\dot{B}_{r,q}^{s_r+\frac{2}{q}})$.
	Then for every $0\leq t\leq T$, 
	\beq
	|\int_0^t\int_{\RR^3}(a\cdot\nabla a)\cdot cdxds|
	\leq \|\nabla a\|_{L^2(\RR_+,L^2)}^2+C\int_0^t\|a(s)\|^2_{L^2}\|c(s)\|^q_{\dot{B}_{r,q}^{s_r+\frac{2}{q}}} ds.
	\enq 
\end{lemma}
Now we recall that for any $T\in \RR_+\cup\{+\infty\}$
\beq
E(T)=L^{\infty}([0,T^*),L^2)\cap L^{2}([0,T^*),\dot{H}^1).
\enq  

\begin{lemma}\label{heat-energy}
	\begin{enumerate}
		\item Let $p>3$ and $T>0$. \\
		      Suppose that $v\in E(T)$ and $\bar{v}\in \mathcal{L}^{\infty}([0,T],\dot{B}^{s_p}_{p,\infty})$ . Then $B(v,\bar{v})\in E(T)$.
		\item Let $T\in(0,\infty)$. Suppose that $v\in L^{\infty}([0,T],L^{3,\infty})$ and $\bar{v}\in L^2([0,T],L^{6,2})$. Then $B(v,\bar{v})\in E(T)$
	\end{enumerate}
\end{lemma}
 \begin{proof}
 We denote  $w:=B(v,\bar{v})$, which satisfies the system
 \beq
\left\{
  \begin{array}{ll}
    \p_t w-\La w+\bar{v}\cdot\nabla v+v\cdot\nabla\bar{v} =-\nabla\pi,\\
    \nabla\cdot w=\nabla\cdot v=\nabla\cdot\bar{v}=0,\\
    w|_{t=0}=w_0
  \end{array}
\right.
\enq

For  $v\in E(T)$ and $\bar{v}\in \mathcal{L}^{\infty}([0,T],\dot{B}^{s_p}_{p,\infty})$, by Proposition 4.2 in \cite{gip}, we obtain that $B(v,\bar{v})\in E(T)$.

Hence we are left with the proof of the second statement of the lemma. We now suppose that $v\in L^{\infty}([0,T],L^{3,\infty})$ and $\bar{v}\in L^2([0,T],L^{6,2})$.
 
 First let $J_{\eps}$ be a smoothing operator that multiplies in the frequency space by a cut-off function bounded by $1$ which is a smoothed out version of the characteristic function of the annulus $\{\eps<|\xi|<\frac{1}{\eps}\}$. 
Then we have for any $t\in[0,T]$
\beq
\begin{aligned}
\|J_{\eps }w(t)\|_{L^2}^2+2\int_0^t\|\nabla J_{\eps }w(s)\|_{L^2}^2ds&=\|w_0\|^2_{L^2}+ 2\int_0^t\int_{\RR^3}(v\cdot\nabla J^2_{\eps}w)\cdot \bar{v}dxds\\
&+2\int_0^t\int_{\RR^3}(\bar{v}\cdot\nabla J^2_{\eps}w)\cdot vdxds.
\end{aligned}
\enq 
Then for any $t\in[0,T]$,
\beq
&&|\int_0^t\int_{\RR^3}(v\cdot\nabla J^2_{\eps}w)\cdot \bar{v}+\int_0^t\int_{\RR^3}(\bar{v}\cdot\nabla J^2_{\eps}w)\cdot vdxds|\\
&\leq&C\int_0^t\|\nabla J^2_{\eps}w\|_{L^2}\|\bar{v}\|_{L^{6,2}}\|v\|_{L^{3,\infty}}\\
&\leq&\frac{1}{2}\int_0^t\|\nabla J_{\eps}w\|_{L^2}^2+\frac{C^2}{2}\|v\|^2_{L^{\infty}(\RR_+,L^{3,\infty})}\|\bar{v}\|^2_{L^2([0,T],L^{6,2})},
\enq 
which implies that for any $t\in[0,T]$
\beq
\|J_{\eps }w(t)\|_{L^2}^2+\int_0^t\|\nabla J_{\eps }w(s)\|_{L^2}^2ds\lesssim\|w_0\|^2_{L^2}+ \|v\|^2_{L^{\infty}(\RR_+,L^{3,\infty})}\|\bar{v}\|^2_{L^2([0,T_1],L^{6,2})}.
\enq 
By taking $\eps\to0$, we have $w\in E(T)$. 
\end{proof}

\begin{lemma}\label{L6}
	Let $p>3$. Suppose that $g\in \mathbb{L}^{3:\infty}_{6,\infty}[T<T^*]$ for some $T^*>0$. then we have $g\in L^2([0,T],L^{6,2}(\RR^3))$ for any $T<T^*$.

\end{lemma}
\begin{proof}

    Suppose that $g$ is a function belonging to $\mathbb{L}^{3:\infty}_{6}[T<T^*]$ for some $T^*>0$. Then for any fixed $T<T^*$, we have that 
    \beq
    \|g\|_{\mathcal{L}^3([0,T],\dot{B}^{s_6}_{6,\infty})}\leq T^{\frac{1}{3}}\|g\|_{\mathcal{L}^{\infty}([0,T],\dot{B}^{s_6}_{6,\infty})}.
    \enq  
    Hence we obtain $g\in\mathcal{L}^{3}([0,T],\dot{B}^{s_6}_{6,\infty})\cap\mathcal{L}^{3}([0,T],\dot{B}^{s_6+\frac{2}{3}}_{6,\infty})$. Since that $s_6<0$ and $s_6+\frac{2}{3}>0$, Then by using Proposition 2.22 in \cite{BCD},
    we have
	\beq
	\|g\|_{\mathcal{L}^3([0,T],\dot{B}^{0}_{6,1})}&\leq& \|g\|^{\frac{1}{4}}_{\mathcal{L}^{3}([0,T];\dot{B}_{6,\infty}^{s_6})}\|g\|^{\frac{3}{4}}_{\mathcal{L}^3([0,T],\dot{B}_{6,\infty}^{s_6+\frac{2}{3}})}\\
	&\leq&T^{\frac{1}{3}}\|g\|_{\mathbb{L}^{3:\infty}_{6,\infty}(T)}.
	\enq

	Now we are left with proving that 
	\beq
	\mathcal{L}^3([0,T],\dot{B}^{0}_{6,1})\hookrightarrow L^3([0,T],L^{6,2}).
	\enq 
	By Littlewood-Paley decomposition,
	\beq
	\|g\|_{L^3([0,T],L^{6,2})}\leq \sum_{j\in\mathbb{Z}}\|\La_j g\|_{L^3([0,T],L^{6,2})}.
	\enq 
	And $\La_j g$ can be written as the following convolution form:
	\beq
	\La_j g=\La_j(\La_j g)=2^{3j}\int_{\RR^3}h(2^j(x-y))\La_j g(y)dy.
	\enq 
	By using Young's inequality,
	\beq
	\|\La_j g\|_{L^3([0,T],L^{6,2})}&\lesssim& 2^{3j}\|h(2^j\cdot)\|_{L_x^1}\|\La_j g\|_{L^3([0,T],L^{6})}\\
	&\lesssim&\|\La_j g\|_{L^3([0,T],L^{6})}\lesssim c_j \|g\|_{\mathcal{L}^3([0,T],\dot{B}^{0}_{6,1})},
	\enq 
	where $\sum_{j\in \mathbb{Z}}|c_j|=1$.
	Then we have 
	\beq
	\|g\|_{L^3([0,T],L^{6,2})}\lesssim \|g\|_{\mathcal{L}^3([0,T],\dot{B}^{0}_{6,1})},
	\enq 
	which combined with the fact that
	\beq
	\|g\|_{L^2([0,T],L^{6,2})}\leq T^{\frac{1}{6}} \|g\|_{L^3([0,T],L^{6,2})}.
	\enq 
	The lemma is proved.
\end{proof}

\subsection*{Acknowledgement}
The author is grateful to Isabelle Gallagher for sharing many hindsights about the Navier-Stokes equation and entertaining discussions for overcoming the difficulties during this research.


\begin{thebibliography}{99}
\bibitem{BCD}H. Bahouri, J.-Y Chemin \& R. Danchin, Fourier Analysis and Nonlinear Partial Differential Equations. {\it Springer-Verlag Berlin Heidelberg}, 2011.

\bibitem{B1} O. A. Barraza, Self-similar solutions in weak $L^p$-spaces of the Navier-Stokes equations. {\it Rev. Mat. Iberoamericana.} 12 (1996), no. 2, 411–439. 

\bibitem{B2} O. A. Barraza, Regularity and stability for the solutions of the Navier-Stokes equations in Lorentz spaces. {\it  Nonlinear Anal.} 35(1999), no. 6, Ser. A: Theory Methods, 747–764. 
\bibitem{BL}J. Bergh \& J. Löfström, Interpolation spaces: an introduction. {\it Springer}, 1976.

\bibitem{bbis}  C. Bjorland, L. Brandolese, D. Iftimie \& M. E. Schonbek, $L^p$-solutions of the steady-state Navier-Stokes equations with rough external forces. {\it Communications in Partial Differential Equations}, 2010, 36(2), 216-246

\bibitem{BP} J. Bourgain \& N.  Pavlović, Ill-posedness of the Navier-Stokes equations in a critical space in 3D
{\it J. Funct. Anal.} 255 (2008), no. 9, 2233–2247.

\bibitem{CC} C. P. Calderón, Existence of weak solutions for the Navier-Stokes equations with initial data in $L^p$. {\it  Trans. Amer. Math. Soc.}318 (1990), no. 1, 179–200.

\bibitem{CM} M. Cannone, A generalization of a theorem by Kato on Navier-Stokes equations. {\it Rev. Mat. Iberoamericana} 13 (1997), no. 3, 515–-541. 

\bibitem{CK} M. Cannone \& G. Karch, About the regularized Navier-Stokes equations. {\it J. Math. Fluid Mech.} 7 (2005), no. 1, 1–28.

\bibitem{CMP} M. Cannone, Y. Meyer \& F. Planchon, Solutions auto-similaires des équations de Navier-Stokes. 
{\it Séminaire sur les Équations aux Dérivées Partielles, 1993–-1994,}  No. VIII, 12 pp., École Polytech., Palaiseau, 1994. 

\bibitem{CP} M. Cannone \& F. Planchon, On the non-stationary Navier-Stokes equations with an external force. {\it Adv. Differential Equations} 4 (1999), no. 5, 697–730.


\bibitem{C} J.-Y Chemin, Théorèmes d'unicité pour le système de Navier-Stokes tridimensionnel. (French) [Uniqueness theorems for the three-dimensional Navier-Stokes system]. {\it  J. Anal. Math.}77 (1999), 27–50. 

\bibitem{C1}J.-Y Chemin, Remarques sur l'existence globale pour le système de Navier-Stokes incompressible. {\it  SIAM J. Math. Anal.} 23 (1992), no. 1, 20–28.

\bibitem{CL} J.-Y Chemin \& N. Lerner, Flot de champs de vecteurs non lipschitziens et équations de Navier-Stokes. {\it J. Differential Equations} 121 (1995), no. 2, 314–328. 


\bibitem{fk} H. Fujita \& T. Kato, On the Navier-Stokes initial value problem. I. 
{\it Arch. Rational Mech. Anal.} 16 1964 269–-315.

\bibitem{gip} I. Gallagher, D. Iftimie \& F. Planchon, Asymptotics and stability for global solutions to the Navier-Stokes equations.
{\it Ann. Inst. Fourier (Grenoble)} 53 (2003), no. 5, 1387–-1424. 

\bibitem{gkp} I. Gallagher, G. S. Koch \& F. Planchon, Blow-up of critical Besov norms at a potential Navier-Stokes singularity. {\it  Comm. Math. Phys.} 343 (2016), no. 1, 39–82.


\bibitem{GP} I. Gallagher \& F. Planchon, On global infinite energy solutions to the Navier-Stokes equations in two dimensions.  {\it  Arch. Ration. Mech. Anal. }161 (2002), no. 4, 307–337. 

\bibitem{G}P. Germain, The second iterate for the Navier-Stokes equation. {\it J. Funct. Anal.} 255 (2008), no. 9, 2248–2264.

\bibitem{Kato} T. Kato, Strong $L^p$ solutions of the Navier-Stokes equation in $\RR^m$, with applications to weak solutions.
{\it Mathematische Zeitschrift}, 187, 1984, pages 471--480.




\bibitem{KT} H. Koch \& D. Tataru, Well-posedness for the Navier-Stokes equations. {\it Adv. Math.} 157 (2001), no. 1, 22–35.

\bibitem{L}J. Leray, Sur le mouvement d'un liquide visqueux emplissant l'espace. {\it Acta Math.} 63 (1934), no. 1, 193–248.

\bibitem{NRS} J. Nečas, M.  Růžička \& V.  Šverák, On Leray's self-similar solutions of the Navier-Stokes equations.

\bibitem{FP}F. Planchon, Asymptotic behavior of global solutions to the Navier-Stokes equations in $\RR^3$. {\it  Rev. Mat. Iberoamericana} 14 (1998), no. 1, 71–93.





\bibitem{Y} T. Yoneda, Ill-posedness of the 3D-Navier-Stokes equations in a generalized Besov space near $\mathrm{BMO}^{-1}$.
{\it  J. Funct. Anal.} 258 (2010), no. 10, 3376–3387.











\end{thebibliography}
\end{document}